\documentclass[12pt]{extarticle}
\usepackage{amsmath, amsthm, amssymb, color}
\usepackage{graphicx}
\usepackage[labelfont = it,textfont=it]{caption}
\usepackage{mathtools}
\usepackage{enumitem}
\usepackage{verbatim}
\usepackage{longtable}
\usepackage{pifont}
\usepackage{makecell}
\usepackage{subcaption}
\usepackage{xcolor}
\usepackage{calligra}
\usepackage{mathrsfs}
\usepackage{tikz}
\usetikzlibrary{matrix}
\usetikzlibrary{arrows}
\usepackage{tikz-cd}
\usetikzlibrary{calc,babel}
\usetikzlibrary{shapes.geometric, arrows}
\tikzstyle{startstop} = [rectangle, rounded corners, minimum width=3cm, minimum height=1cm,text centered, draw=black]
\tikzstyle{arrow} = [thick,->,>=stealth]

\usepackage[colorlinks,plainpages,hypertexnames=false,plainpages=false]{hyperref}
\hypersetup{urlcolor=black, citecolor=black, linkcolor=black}

\oddsidemargin \evensidemargin
\usepackage{makecell}
\usepackage{multirow,array}

\newtheorem{theorem}{Theorem}
\newtheorem*{satz*}{Satz}
\newtheorem{proposition}[theorem]{Proposition}
\newtheorem{lemma}[theorem]{Lemma}
\newtheorem{corollary}[theorem]{Corollary}

\theoremstyle{definition}
\newtheorem{definition}[theorem]{Definition}

\newtheorem{remark}[theorem]{Remark}

\newtheorem{example}[theorem]{Example}

\newtheorem{notation}[theorem]{Notation}
\newtheorem{strategy}[theorem]{Strategy}
\newtheorem{question}[theorem]{Question}
\newtheorem*{claim*}{Claim}
\newtheorem*{thm:bn_strata}{Theorem \ref{thm:bn_strata}}
\newtheorem*{cor:spectral_curves}{Corollary \ref{cor:closure_stable_locus}}
\numberwithin{theorem}{section}

\newcommand{\thistheoremname}{}
\newtheorem*{genericthm*}{\thistheoremname}
\newenvironment{namedthm*}[1]
{\renewcommand{\thistheoremname}{#1}%
	\begin{genericthm*}}
	{\end{genericthm*}}
\newenvironment{namedtheorem*}[1]
{\renewcommand{\thistheoremname}{#1}%
	\begin{genericthm*}}
	{\end{genericthm*}}

\newcommand{\PP}{\mathbb{P}}

\newcommand{\NN}{\mathbb{N}}

\newcommand{\Aff}{\mathbb{A}}

\newcommand{\cM}{\mathcal{M}}
\newcommand{\OO}{\mathcal{O}}
\newcommand{\cH}{\mathcal{H}}
\newcommand{\cE}{\mathcal{E}}
\newcommand{\cL}{\mathcal{L}}
\newcommand{\cV}{\mathcal{V}}
\newcommand{\cF}{\mathcal{F}}
\newcommand{\cD}{\mathcal{D}}

\newcommand{\isomarrow}{\xrightarrow{\sim}}
\newcommand{\op}{\operatorname}

\DeclareMathOperator{\gSpec}{\mathcal{S}pec}

\DeclareMathOperator{\Hom}{\mathcal{H}om}
\DeclareMathOperator{\Ext}{\mathcal{E}xt}
\DeclareMathOperator{\End}{\mathcal{E}nd}

\renewcommand{\tilde}{\widetilde}		%never use ordinary tilde
\let\div\relax							%replace obscure division symbol with principal divisor div
\DeclareMathOperator{\div}{div}

\tikzset{
  curve/.style={
    settings={#1},
    to path={
      (\tikztostart)
      .. controls ($(\tikztostart)!\pv{pos}!(\tikztotarget)!\pv{height}!270:(\tikztotarget)$)
      and ($(\tikztostart)!1-\pv{pos}!(\tikztotarget)!\pv{height}!270:(\tikztotarget)$)
      .. (\tikztotarget)\tikztonodes
    },
  },
  settings/.code={%
    \tikzset{quiver/.cd,#1}%
    \def\pv##1{\pgfkeysvalueof{/tikz/quiver/##1}}%
  },
  quiver/.cd,
  pos/.initial=0.35,
  height/.initial=0,
}

\title{Towards Brill-Noether Theory for Spectral Curves}

\author{Clemens Nollau}

\begin{document}
\maketitle
\begin{abstract}
We study Brill-Noether loci of three kinds of spectral curves: classical spectral curves as introduced by Hitchin, spectral curves over the projective line and double covers whose branch locus is a canonical divisor. Our techniques are based on the Beauville-Narasimhan-Ramanan correspondence: We push down line bundles on the spectral curve to the base curve and then we study the Higgs bundles obtained in this way.

For the first kind we study the spaces of pencils in the Picard variety of a classical spectral curve in detail. In the case of spectral curves over the projective line we deal with their splitting loci which refine the Brill-Noether loci in the Picard variety. We compute their dimensions and investigate whether they are smooth.  For the third kind we determine the gonality sequence when the rank of the linear system is much smaller than the genus. For this the base curve and the branch divisor are assumed to be general.   
\end{abstract}
\section{Introduction}
The study of curves and their linear systems is a classical topic in algebraic geometry and has been a driving force in its modern developement. The description of all complete linear systems on a general curve in $\cM_g$ began in the 19th century but rigorous proofs of the Brill-Noether theorems were only available in the 1980s. For an in-depth treatment see \cite{ACGH}.

Today one also seeks to describe the spaces $W^r_d(C)$ on curves $C$ which are very much not general. Recently the Brill-Noether theorems have been extended to curves $C$ together with a branched cover $ \psi: C \to \PP^1$ of degree $n$ by Hannah Larson in \cite{Hlarson}. More precisely, $\psi: C\to \PP^1$ should be a general element of the Hurwitz space $\cH_{g,n}$. She introduced the splitting loci of $C$: 
\begin{equation*}
U^{(e,m)}(C) = \left\{ L \in \op{Pic}(C): \psi_*L \cong \bigoplus_{j=1}^l \OO_{\PP^1}(e_j)^{\oplus m_j} \right\}. 
\end{equation*}
The splitting type $(e,m)$ is the vector $(e_1,\ldots,e_l;m_1,\ldots,m_l)$.  They stratify the classical Brill-Noether loci because every vector bundle over $\PP^1$ splits as a direct sum of line bundles. We remark that our labeling of splitting loci differs from the one used in \cite{Hlarson}. The main result of \cite{Hlarson} and \cite{LLV} can be summarized as follows:
\begin{theorem}\cite[Larson, Larson, Vogt]{Hlarson,LLV}\label{thm:hurwitz_brill_noether}
Let $\psi: C \to \PP^1$ be a general cover of degree $d$ and genus $g$. Then $U^{(e,m)}(C)$ has dimension  
\begin{equation*}
\rho(g,e,m) := g(C) - \sum_{1 \leq i < j\leq l}(e_j - e_i -1 )m_im_j.
\end{equation*}
if $\rho(g,e,m)$ is non-negative. Otherwise the splitting locus is empty. Furthermore $U^{(e,m)}(C)$ is irreducible if $\rho(g,e,m) > 0$.
\end{theorem} 
It is also interesting to look at other classes of special curves  and attempt a similar description. Here we start treating spectral curves. These are curves $\tilde{C}$ which are cut out by a single equation in the global space $\Aff(N)$ of a line bundle $N$ over a curve $C$. Our approach rests upon the Hitchin correspondence between line bundles on $\tilde{C}$ and Higgs bundles with a fixed characteristic polynomial on $C$. All we need to know about spectral curves  will be reviewed in section \ref{section:prelim}. 

In Section \ref{section:rat_base} we deal with spectral covers $ \psi: C \rightarrow \PP^1$ over the projective line. Inspired by Theorem \ref{thm:hurwitz_brill_noether} we prove the following statement about these loci:
\begin{namedthm*}{Theorem A}\label{thm:dim_split_loci}
\emph{
Let $(s_1,\ldots,s_n)$ be a general element of $\bigoplus_{i=1}^n H^0(\OO_{\PP^1}(ik))$ and denote by $\tilde{C}$ the associated spectral curve. Write $V$ for the vector bundle $\bigoplus_{j=1}^l \OO_{\PP^1}(e_j)^{\oplus m_j}$.
Then 
\begin{equation}\label{eq:dim_bn_loci}
\dim U^{(e,m)}(\tilde{C}) \leq h^0(\End(V) \otimes \OO(k)) + 1-  \sum_{1\leq i \leq j \leq l}(e_j - e_i+1)m_im_j - d - k\binom{n+1}{2}
\end{equation}
holds for a splitting type $(e,m)$ if the right hand side is nonnegative. This is either an equality or the splitting locus $U^{(e,m)}(\tilde{C})$ is empty.
Take $(e,m)$ to be a splitting type satisfying $e_1 - e_l + k\geq -1$.   We have
\begin{equation}\label{eq:dim_split_loci}
\dim U^{(e,m)}(\tilde{C}) = g(\tilde{C}) - \sum_{1 \leq i < j\leq l}(e_j - e_i -1 )m_im_j 
\end{equation}
if the right hand side is nonnegative.
Furthermore in this case the space $U{(e,m)}(\tilde{C})$ is smooth.
}
\end{namedthm*} 
While our result gives a weaker result about general covers of a given degree it has the advantage of being quite simple: Instead of degenerating to a chain of elliptic curves like in \cite{Hlarson} we work directly with smooth curves. For a more precise comparison see Remark \ref{remark:larson comparison}.

In the rest of the paper we study branched covers where the base curve has positive genus. In section \ref{section:classical_spectral} we treat spectral curves as introduced by Hitchin in \cite{Hitchin}. They are spectral curves $\tilde{C}$ of degree $2$ over a curve $C$ defined with respect to the line bundle $N = K_C$. We refer to them as \textbf{classical spectral curves}. Inspired by the splitting loci which we have just introduced we stratify the spaces of pencils $W^1_d(\tilde{C})$. If $L$ is a line bundle on the spectral curve with $h^0(\tilde{C},L) \geq 2$ , we distinguish cases based on whether the pushforward of $L$ to $C$ is stable, strictly semistable or unstable. Diagram \ref{fig:stratification} illustrates this approach. 
Note that if $g$ is the genus of $C$ then $\tilde{C}$ has genus $4g-3$. Hence all line bundles in $\op{Pic}^{4g-2}(\tilde{C})$ have two linearly independent global sections. We find the dimensions of all strata in the range $0 \leq d \leq 4g-3$. Stating the result turns out to be quite involved so we restrict ourselves in the introduction to the most important stratum: Consider the open subscheme $W^1_d(\tilde{C})_s$ of $W^1_d(\tilde{C})$ consisting of line bundles whose pushforward is stable. In Theorem \ref{thm:bn_strata}, we prove among other things:
\begin{namedthm*}{Theorem B}
\emph{
Let $C$ be a general curve of genus $g\geq 3$. Let $\psi: \tilde{C} \to C$ be a spectral cover corresponding to a general element $(s_1,s_2)$ of the Hitchin base $H^0(K_C)\oplus H^0(K_C^2)$.
The stratum $W^1_d(\tilde{C})_s$ is empty if $d < 2g+1$. Assume $2g+1 \leq d \leq 4g-3$. Then the dimension of $W^1_d(\tilde{C})_s$ is $2d - 4g +1$. This is the Brill-Noether number $\rho(g(\tilde{C}),1,d)$.}
\end{namedthm*}
\begin{remark}
	Note that for $d = 2g$ we have $2d -4g +1 = 1$ but $W^1_d(\tilde{C})$ is still empty. The explanation for this is the following: Pushing forward a line bundle on $\tilde{C}$ of degree $2g$ gives a vector bundle $V$ of rank $2$ and degree $2$. Then $V$ can not be stable and satisfy $h^0(C,V)\geq 2$ by \cite{Teixidor}.
\end{remark} 
As a consequence of describing the dimensions of all strata we obtain for $\tilde{C}$ as in Theorem B the following:
\begin{corollary}
The gonality of $\tilde{C}$ is $2\lfloor\frac{g+3}{2}\rfloor$. The equality $\dim W^1_d(\tilde{C}) = \rho(g(\tilde{C}),1,d)$ holds if and only if $3g-3 \leq d \leq 4g-2$. If $3g-3 < d $, then $W^1_d(\tilde{C})$ is the closure of $W^1_d(\tilde{C})_s$. 
\end{corollary}
Hence we can say that $W^1_d(\tilde{C})$ behaves as predicted by classical Brill-Noether theory if $d \geq 3g-3$. 
To prove that the stratum $W^1_d(\tilde{C})_s$ is actually nonempty for a general spectral curve $\tilde{C}$ requires some technical work which is done in section \ref{section:restricted_hitchin}. It turns out that we need to consider for certain vector bundles $V$ of rank $2$ on $C$ the following map:
\begin{equation*}
h_V : H^0(\End_0(V)\otimes K_C) \to H^0(K_C^2), \phi \mapsto \det \phi.
\end{equation*}
Here $\End_0(V)$ denotes the trace-free endomorphisms of $V$. We refer to $h_V$ as the \textbf{restricted Hitchin map}. We give partial answers to the following question
\begin{question}
Is $h_V$ dominant for all stable vector bundles $V$ of rank $2$ on a curve $C$ of genus $g \geq 2$?
\end{question}
The approach in section \ref{section:restricted_hitchin} might be of independent interest in the study of the Hitchin system.

In section \ref{section:gonality} we look at spectral curves defined with respect to a theta characteristic instead of the canonical line bundle. The double covers we obtain can be defined as follows:
\begin{definition}
A double cover between smooth curves $\psi: \tilde{C} \rightarrow C$ is called a \textbf{canonical cover} if the branch divisor of $\psi$ is a canonical divisor. 
\end{definition}

For the source curve $\tilde{C}$ of a general canonical cover, we determine the gonality and more generally the gonality sequence $d_1,d_2,\ldots$. Here $d_r$ is the smallest integer $d$ such that a $g^r_d$ exists on $\tilde{C}$. The gonality $\op{gon}(C) = d_1$ is related to the study of syzygies via the Green-Lazarsfeld conjecture. We determine these invariants for canonical covers where the target curve and the branch divisor are general in most of the cases: 
\begin{namedthm*}{Theorem C}\label{thm:gonality_seq}
\emph{
Let $\psi : \tilde{C} \rightarrow C$ be a canonical cover.
Assume the curve $C$ to be Brill-Noether general of genus $g$ 
\begin{itemize}
	\item[(i)] Assume $g\geq 10$. Then $\tilde{C}$ has gonality $2\op{gon}(C)$ and $W^1_{\op{gon}(\tilde{C})}(\tilde{C})$ equals $\psi^*W^1_{\op{gon}(C)}$.
	\item[(ii)] Let $\psi: \tilde{C} \rightarrow C$ be a canonical cover where $C$ is assumed to be general of genus $g$. Denote the gonality sequences of $\tilde{C}$ and $C$ by $(\tilde{d_i})$ and $(d_i)$ respectively. If $1 < r \ll g$ holds then we have $\tilde{d}_r = 2d_r$ and furthermore $W^r_{\tilde{d}_r}(\tilde{C}) = \psi^*W^r_{d_r}(\tilde{C})$.
\end{itemize}
}
\end{namedthm*}
This first part will be proved in a more precise version in section \ref{section:gonality}. In the subsequent section we will use classical Brill-Noether theory to show the second part of the Theorem:

The last section gives examples of canonical covers and their equations in low genus. As a highlight we construct equations of canonical covers $\tilde{C}$ of genus $7$ with gonality $4$ and $5$ respectively. This serves as an illustration how to make our Higgs bundle approach explicit in some cases.    
\subsection*{Acknowledgements} 
I would like to thank Daniele Agostini for suggesting this project and being a very encouraging advisor. Furthermore I thank Hannah Larson and Sameera Vemulapalli for sharing their own work and pointing out a mistake in a previous version of this paper. Thanks to Feiyang Lin and Daniel Funck for helpful discussions.
\section{Preliminaries on Spectral Curves}\label{section:prelim}

Throughout the text $\Bbbk$ denotes a fixed algebraically closed field of characteristic zero. We use Grothendieck's convention for projectivization: If $V$ is a vector space the points of $\PP V$ are one dimensional quotients of $V$. We recall the framework for working with spectral curves. 

\begin{definition}
Let $C$ be a smooth curve and $n$ a natural number. Let $N$ be an ample line bundle on $C$ and $s_i \in H^0(C,N^i)$ for $i=1,\ldots,n$. We call $(C,N,s_1,\ldots,s_n)$ the data for a spectral curve. 
\end{definition}

Here is how to construct a spectral curve $\tilde{C}$ out of this data.  Take $\Aff(N) := \gSpec_C\op{Sym}^\bullet(N^{-1})$ to be the global space of the line bundle $N$. It comes with a map $\pi: \Aff(N) \to C$. By the projection formula we have 
\begin{equation*}
\pi_*\OO_{\Aff(N)} \cong \bigoplus_{i\geq 0}N^{-i},\qquad \pi_*\pi^*N^m \cong \bigoplus_{i \geq 0} N^{m-i}.
\end{equation*}
We denote the section $1 \in H^0(\OO_C) \subset H^0(\pi^* N)$ by $\tau$. 
\begin{definition}
Let $(C,N,s_1,\ldots,s_n)$ be the data for a spectral curve. The associated spectral curve is defined by 
\begin{align*}
\tilde{C} := V(\tau^n + \sum_{i=1}^n s_{d-i}\tau^i) \subset \Aff(N).
\end{align*}
The restriction of the natural projection $\Aff(N) \rightarrow C$ to $\tilde{C}$ will be called the spectral cover $\psi$.
\end{definition}
\begin{remark}
 The data $(C,N,(s_i))$ and $(C,N',(s_i'))$ define an equivalent spectral curve if there is an isomorphism $\iota: N \rightarrow N'$ and a nonzero scalar $\lambda \in \Bbbk$ such that $\iota(s_i') = \lambda^i s_i$.
\end{remark}
 Now take any line bundle $M$ on $C$. The following applications of the projection formula will be used very often:
\begin{align}\label{eq:push_and_pull}
\psi_*\OO_{\tilde{C}} \cong \bigoplus_{j=0}^{n-1}N^{-j},\qquad\psi_*\psi^*M\cong  \bigoplus_{j=0}^{n-1} M \otimes N^{-j}.
\end{align}
For an arbitrary line bundle $L$ on $\tilde{C}$ its pushforward is a vector bundle of rank $n$ on $C$. Its determinant is given by
\begin{equation}\label{eq:push_and_det}
\det \psi_* L \cong \op{Nm}(L) \otimes \det(\psi_*\OO_{\tilde{C}})^\vee \cong \op{Nm}(L)\otimes N.
\end{equation}
Here the norm is defined as follows: If $L = \OO_C(D)$ for a divisor $D$, then $\op{Nm(L)}:= \OO(f_*D)$. See for example \cite[IV, Exercise 2.6]{Hartshorne}. As a consequence the degree of $\psi_*L$ is
\begin{align}\label{eq:deg_push}
\deg \psi_*L =\deg L - \binom{n}{2}\deg N 
\end{align}
One way to get the genus of the spectral curve is to compute $h^1(\tilde{C},\OO_{\tilde{C}})$ via equation (\ref{eq:push_and_pull}):
\begin{equation}
g(\tilde{C}) = n(g(C) -1) + 1 + \binom{n}{2}\deg N.
\end{equation}

Our approach for understanding line bundles $L$ on $\tilde{C}$ is based on studying Higgs bundles of rank $n$ on $C$ as we explain now. The map $L \to L \otimes \psi^*N$ induced by multiplication with $\tau_{|\tilde{C}}$ pushes down to a so called twisted endomorphism
\begin{equation*}
\phi: \psi_*L \to \psi_*L\otimes N.
\end{equation*}
From the data $(\psi_*L,\phi)$ one can recover the line bundle $L$. To state this result we need a definition. 
\begin{definition}
Let $N$ and $V$ be vector bundles on a smooth curve $C$ of rank $1$ and $n$ respectively. Let $\phi$ be a twisted endomorphism of $V$ with respect to $N$ i.e, a map $\phi: V \to V\otimes N$. The vector bundle $V$ together with $\phi$ will be called an $N$-valued \textbf{Higgs bundle}.  Two Higgs bundles $(V,\phi)$ and $(V',\phi')$ with respect to $N$ will be considered equivalent if there is an isomorphism $\nu: V \rightarrow V'$ such that the following diagram commutes:
\begin{equation*}
\begin{tikzcd}[ampersand replacement=\&]
	V \& {V\otimes N} \\
	{V'} \& {V'\otimes N}
	\arrow["\phi", from=1-1, to=1-2]
	\arrow["{\phi'}", from=2-1, to=2-2]
	\arrow["\nu", from=1-1, to=2-1]
	\arrow["{\nu\otimes id_{N}}", from=1-2, to=2-2]
\end{tikzcd}
\end{equation*}
For a Higgs bundle $(V,\phi)$ we define the characteristic polynomial $\chi_\phi$ to be
\begin{align*}
\chi_\phi  = T^n + \sum_{i=0}^{n-1} s_{n-i} T^i, \quad \text{with $s_i = (-1)^i\op{tr}\left(\wedge^i\phi\right) \in H^0(C,N^i)$.}
\end{align*}
\end{definition}
We next single out a class of well-behaved Higgs bundles:
\begin{definition}
A Higgs bundle $(V,\phi)$ with respect to a line bundle $N$ on a curve $C$ will be called \textbf{stable} if $V$ does not admit a proper subbundle $M$ satisfying $\phi(M) \subset
 M \otimes N$ and $\mu(V) \leq \mu(M)$. Here $\mu(M)$ is the slope $\frac{\deg M}{\op{rk}M}$.
\end{definition}
\begin{remark}
	Recall the distinction between subsheaves and subbundles of a vector bundle $V$ on a smooth curve. We say that a subsheaf $M$ of $V$ is a subbundle if and only if $V/M$ is locally free. 
\end{remark}
The next result is sometimes called the Hitchin correspondence or Beauville- Narasimhan-Ramanan correspondence (BNR) after the pioneerring works \cite{Hitchin} and \cite{BNR}. 
\begin{proposition}\label{prop:hitchin_corr}
Let $(C,N,s_1,\ldots,s_n)$ be the data for a spectral cover $\psi: \tilde{C} \to C$ such that $\tilde{C}$ is smooth. There is a bijective correspondence between line bundles $L$ on $\tilde{C}$ and those Higgs bundles of rank $n$ $(V,\phi)$ with respect to $N$ which have characteristic polynomial $T^n + \sum_{i=0}^{n-1}s_iT^i$. The bijection is given by sending $L$ to $\psi_*L$. The map $\phi$ is induced by multiplication with $\tau_{|C}$.

Furthermore if $(V,\phi)$ comes from a line bundle $L$ on $\tilde{C}$ then $\phi$ leaves invariant no proper subbundle of $V$, in particular it is a stable Higgs bundle. 
\end{proposition}
\begin{proof}
See \cite[Proposition 3.6]{BNR}. The second assertion is \cite[Lemma 5.1]{NagarajRamanan}.
\end{proof}
The following relation between twisted endomorphisms and usual endomorphisms of a vector bundle will be useful to us in the last section.
\begin{lemma}\label{lemma:aut}
Let $(V,\phi)$ be a Higgs bundle with respect to a line bundle $N$. Assume that the spectral curve $\tilde{C}$ defined by $\chi_\phi$ is irreducible. Then the following holds: \begin{equation*}
\{ \rho \in \op{End}(V) : (\rho \otimes \op{id}_N)\circ\phi = \phi \circ \rho \} = \{ \lambda  \op{id}_V : \lambda \in \Bbbk\}.
\end{equation*}
The same conclusion holds if $(V,\phi)$ is a stable Higgs bundle.  
\end{lemma}
\begin{proof}
The proof is based on \cite[Lemma 5.2]{NagarajRamanan}. Take $\rho$ to be an endomorphism of $V$ which commutes with $\phi$. Because $\rho$ is an endomorphim its  characteristic polynomial has constant coefficients. Hence $\rho$ has an eigenvalue $\lambda \in \Bbbk$. Then $\rho - \lambda\op{id}_V$ commutes with $\phi$ as well. Write $V_1$  for the sub vector bundle $\op{ker}(\rho - \lambda\op{id})$. It is invariant under $\phi$ and we have two induced twisted endomorphisms $\phi': V_1 \rightarrow V_1 \otimes N$ and $\phi'': V/V_1 \rightarrow V/V_1 \otimes N$. In terms of characteristic polynomials we get
\begin{equation*}
\chi_\phi = \chi_{\phi'}\chi_{\phi''}.
\end{equation*}
Because $\chi_\phi$ cuts out $C$ in $\Aff(N)$ it is irreducible and we conclude $V = V_1$. Hence $\rho = \lambda\op{id}_V$.

If $(V,\phi)$ is stable, we obtain $V_1$ as before. Because $V_1$ is invariant under $\phi$ we have $\mu(V_1) < \mu(\phi)$. But $V/V_1 \cong \op{im}(\rho-\lambda \op{id}_V)$ is also $\phi$-invariant. If this subsheaf were nonzero, we would we have $\mu(V/V_1) < \mu(V)$. This contradicts $\mu(V_1) < \mu(V)$ and we conclude $V = V_1$. 
\end{proof}
Let us define the moduli problem for Higgs bundles. 
\begin{definition}
Let $C$ be a curve and $N$ a line bundle on it. Furthermore fix positive integers $n$ and $d$. Define the \textbf{Higgs-bundle functor} $Sch_\Bbbk  \rightarrow Set$ for the line bundle $N$, rank $n$ and degree $d$ as follows. A scheme $T$ is mapped to the set of equivalence classes of pairs $(V_T,\phi_T)$. Here $V_t$ is a vector bundle of rank $n$ and degree $d$ on $C \times T$ and $\phi_T$ is a morphism of coherent sheaves $V_T \rightarrow V_T \otimes \pi_C^*N$. We demand that for all closed points  $t \in T$ the pair
\begin{equation*}
(V_{T|t},\phi_{|t}: V_{T|t} \to V_{T|t} \otimes N)
\end{equation*}
is a stable Higgs bundle with respect to $N$.
Two pairs $(V_T,\phi_T),(V'_T,\phi'_T)$ are said to be equivalent if for all $t \in T$ $(V_t,\phi_t) \cong (V'_t,\phi_t')$ holds. The restriction maps are given by pulling back vector bundles and twisted endomorphisms.
\end{definition}
\begin{theorem}
There is a coarse moduli space for the Higgs-bundle functor with respect to $(N,n,d)$, which we denote by $\cH_C(N)(n,d)$.
\end{theorem}
\begin{proof}
See \cite[Theorem 5.10]{Nitsure}.
\end{proof}
We now define Brill-Noether loci in this moduli space:
\begin{definition}
Take $C, N, n$ and $d$ as before and $r$ to be a  nonnegative integer. Define the \textbf{Brill-Noether locus} $W^r_C(N)(n,d)$ to be the locus in $\mathcal{H}_C(N)(n,d)$ consisting of Higgs-bundles $(V,\phi)$ satisfying $h^0(V) \geq r +1$.
\end{definition}
The stability of a Higgs bundle $(V,\phi)$ does not imply that $V$ is stable itself. Nonetheless we can find bounds on how unstable the underlying vector bundle can be. We will only need a result like this in rank $2$:
\begin{lemma}\label{lemma:max_subbundle}
Let $C$ be a smooth curve and $(V,\phi)$ be a stable Higgs bundle of rank $2$ with respect to a line bundle $N$. If $M$ is a line subbundle of $V$ then $M^\vee \otimes V/M \otimes N$ is effective. In particular the following holds 
\begin{equation}\label{eq:max_subbundle}
\deg M \leq \frac{1}{2}\left(\deg V + \deg N\right).
\end{equation} 
In case of equality the Higgs bundle $(V,\phi)$ corresponds to the line bundle of the form $\psi^*M$ on $\tilde{C}$. Here $\psi : \tilde{C} \rightarrow C$ is the spectral cover associated to $(V,\phi)$.
\end{lemma} 
\begin{proof}
We can assume without loss of generality that $\deg M \geq \frac{1}{2}\deg V$. By stability of $(V,\phi)$, the subbundle $M$ is not mapped into $M \otimes N$ by $\phi$. Call $M'$ the line bundle $V/M$. We get a nonzero map
\begin{equation}\label{eq:composition}
\begin{tikzcd}[ampersand replacement=\&]
	M \& V \& {V\otimes N} \& {M'\otimes N}
	\arrow[from=1-1, to=1-2]
	\arrow[from=1-2, to=1-3]
	\arrow[from=1-3, to=1-4]
\end{tikzcd}.
\end{equation}
This implies that $M^\vee \otimes M' \otimes N$ is effective and has degree $\geq 0$. Combining this with $\deg E = \deg M + \deg M'$ the inequality follows. 
For the last part observe that if we have equality in (\ref{eq:max_subbundle}) the map in (\ref{eq:composition}) is an isomorphism. Now the inclusion of $M$ in $V$ has the left inverse $V \rightarrow V \otimes N \rightarrow M' \otimes N \cong M$. Hence $V = M \oplus M'$. We write the twisted endomorphism as a matrix:
\begin{equation}
\phi = \begin{pmatrix}
\phi_{11} & \phi_{12} \\
\phi_{21} & \phi_{22}
\end{pmatrix} : M \oplus (M \otimes N^\vee) \rightarrow (M \otimes N) \oplus M.
\end{equation}
The characteristic polynomial is computed as follows: $-s_1 = \phi_{11} + \phi_{22}, s_2 = \phi_{11}\phi_{22} - \phi_{12}\phi_{21}$. The entry $\phi_{21} \in H^0(\OO_C)$ is nonzero because otherwise the subbundle $M$ is $\phi$-invariant. By conjugating $\phi$ with the automorphism $\begin{pmatrix}
1 & -\phi_{22}/\phi_{21} \\
0 & 1
\end{pmatrix}$ we can assume that $\phi_{11}$ vanishes. Conjugating with $\begin{pmatrix}
\phi_{21} & 0 \\
0 & 1
\end{pmatrix}$ we end up with the twisted endomorphism $\phi = \begin{pmatrix}
0 & -s_2 \\
1 & -s_1
\end{pmatrix}$. One checks that $(V,\phi)$ corresponds under the Hitchin-correspondence to $\psi^*M$.
\end{proof}
We will frequently deal with extensions of line bundles. To bound the global sections of such an extension the following technique due to Green and Lazarsfeld is useful. The author learned it from \cite[Section 3]{Lazarsfeld}. We first introduce some notation. For two line bundles $M,M'$ on a variety $X$ and an extension $\xi$ of $M'$ by $M$ we get an induced vector bundle $V_\xi$. The short exact sequence induces a long exact sequence in cohomology:
\begin{equation*}
\begin{tikzcd}
	0 & {H^0(M)} & {H^0(V_\xi)} & {H^0(M')} & {H^1(M)} & \ldots
	\arrow[from=1-1, to=1-2]
	\arrow[from=1-2, to=1-3]
	\arrow[from=1-3, to=1-4]
	\arrow["{\delta(V)}", from=1-4, to=1-5]
	\arrow[from=1-5, to=1-6]
\end{tikzcd}
\end{equation*}
From this we get a linear function
\begin{align}
\label{eq:boundary_map}
\rho: \op{Ext}^1(M',M) &\rightarrow \op{Hom}(H^0(M'),H^1(M))\\
\xi &\mapsto \delta(\xi).
\end{align}
We denote its kernel by $K_{M,M'}$. Observe that it is the space of all extensions $\xi$ with $h^0(V_\xi) = h^0(M) + h^0(M')$. 
\begin{lemma}\label{lemma:ext}
Let $C$ be a smooth curve and $M, M' $ be line bundles on it. Consider the multiplication map 
\begin{equation*}
\mu: H^0(K_C \otimes M^\vee) \otimes H^0(M') \rightarrow H^0(K_C \otimes M^\vee \otimes M').
\end{equation*}
We have $\dim K_{M,M'} = \op{corank} \mu$. In particular if $\xi$ is not zero and $\mu$ is surjective then $h^0(V_\xi) < h^0(M) + h^0(M')$ holds. 
\end{lemma}
\begin{proof}
The first assertion implies the second. Identifying $\op{Ext}$ with sheaf cohomology we view $K_{M,M'}$ as the kernel of   
\begin{equation*}
\rho : H^1(M \otimes (M')^\vee) \rightarrow \op{Hom}(H^0(M'), H^1(M)) \cong H^0(M')^\vee \otimes H^1(M).
\end{equation*}
The dual of $\rho$ can be identified via Serre duality with $\mu$. Hence $\dim K_{M,M'} = \op{corank} \mu$.
\end{proof}
The next Lemma is a variation of the previous one. We denote the base locus of $f$ by $\op{Bs}(f)$.
\begin{lemma}\label{lemma:lift_section}
Let $C$ be a smooth curve and $M, M' $ be line bundles on it. Let $s \in H^0(M')$ be a non-zero section with a reduced divisor $D$. Consider a non-trivial extension $\xi$ of $M'$ by $M$. Let $f: C \dashrightarrow \PP H^0(M^\vee \otimes M'\otimes K_C)$ be the map induced by the line bundle $M^\vee \otimes M'\otimes K_C$. The section $s$ lifts to a section of $V_\xi$ if and only if $[\xi] \in \PP H^0(M^\vee\otimes M' \otimes K_C)$ lies in the span of $f(D\setminus \op{Bs}(f))$. 
\end{lemma}
\begin{proof}
The boundary map of the extension $H^0(M') \to H^1(M)$ is given by the cup-product: $t \mapsto t \cup \xi$. The section $s$ lifts to a section of $V_\xi$ if and only if $s \cup \xi = 0$. Now we look at the following short exact sequence:
\begin{equation*}
0 \to M\otimes (M')^\vee \xrightarrow{\cdot s} M \to M_{|D} \to 0.
\end{equation*} 
Consider the following portion of its exact sequence:
\begin{equation*}
H^0(M_{|D}) \xrightarrow{\delta} H^1(M\otimes (M')^\vee) \to H^1(M) \to 0.
\end{equation*}
The second map is given by $\zeta \in H^1(M\otimes (M')^\vee) \mapsto s \cup \zeta$. Hence $s \cup \xi = 0$ if and only if $\xi$ there is $ g \in H^0(M_{|D})$ such that $\delta(g) = \xi$. The element $\delta(g)$ is the following functional in $H^1(M\otimes (M')^\vee) \cong H^0(M^\vee\otimes K_C(D))^\vee$:
\begin{equation*}
\delta(g) : \omega \in  H^0(M^\vee\otimes M'\otimes K_C) \mapsto \op{res}_D(\omega g). 
\end{equation*}
By the definition of the base locus we have $\op{res}_D (\omega g) = \op{res}_{D\setminus \op{Bs}(f)} (\omega g)$. Hence $[\delta(g)]$ is in the span of $f(D\setminus \op{Bs}(f))$. Conversely if $[\xi]$ is in the span of $D\setminus \op{Bs}(f)$, we get that the functional $H^0(M^\vee \otimes M' \otimes K_C)$ is a linear combination of the residual functionals at the points of $\op{D}\setminus \op{Bs}(f)$. This gives us the section $g \in H^0(M_{|D})$ we need.
\end{proof}
We will need to bound the dimension of the space of twisted endomorphisms ${H^0(\End(V) \otimes N)}$ of a rank $2$ vector bundle $V$ several times. To achieve this we use 
\begin{lemma}\label{lemma.twisted_bound}
Let $C$ be a smooth curve and $L,M_1,M_2$ line bundles on it. Furthermore let $V$ be an extension of $M_2$ by $M_1$. Then we have 
\begin{equation*}
h^0(\End(V) \otimes L) \leq 2h^0(L) + h^0(M_i \otimes M_j^\vee \otimes L) + h^0(M_j \otimes M_i^\vee \otimes L).
\end{equation*}
If $\deg M_1 > \deg M_2$ we have an equality:
\begin{equation}\label{eq:dim_endos}
\dim \op{End}(V) = \begin{cases}
2 + h^0(C,M_1^\vee\otimes M_2) & \text{if the extension splits,}\\
1 + h^0(C,M_1^\vee\otimes M_2)&\text{else.}
\end{cases}
\end{equation}
\end{lemma}
\begin{proof}
There is a vector bundle $\cV$ on $C \times \Aff^1$ such that $\cV_t \cong V$ for $t \neq 0$ and $\cV_0 \cong M_1 \oplus M_2$. Furthermore $\cE$ is flat over $\Aff^1$.  The above inequality is actually an equality if the extension splits. Because also $\End(\cV) \otimes \op{pr}_1^*L $ is flat over $\Aff^1$ the result follows by the semicontinuity theorem for cohomology.  

For the second part it suffices to consider $V$ which does not split. We consider ${H^0(M_1 \otimes M_2^\vee)}$ as a subspace of $\op{End}(V)$ by sending $s \in H^0(M_1\otimes M_2^\vee)$ to ${V \to M_2 \xrightarrow{s} M_1 \to V} $. Take $\rho \in \op{End}(V)$. Because $H^0(M_1^\vee\otimes M_2) = 0$ the line bundle $M_1$ is stable under $\rho$. Therefore we is a $\lambda\in \Bbbk$ such that $\rho_{|M_1} = \lambda\op{id}_{M_1}$. Consider $\tilde{\rho} = \rho- \lambda\op{id}_V$. The kernel of $\tilde{\rho}$ contains $M_1$ hence $\tilde{\rho}$ descends to a map $\sigma: M_2 \to V$. Because the sequence does not split, the image of $\sigma$ is in $M_1$ i.e $\sigma \in H^0(M_1 \otimes M_2^\vee)$. This shows that $H^0(M_1 \otimes M_2^\vee)$ and $\op{id}_V$ generate $\op{End}(V)$. Because $\op{id}_V$ is not in $H^0(M_1\otimes M_2^\vee)$ the second case of equation (\ref{eq:dim_endos}) follows.
\end{proof} 
We now outline our approach to Brill-Noether theory via Higgs bundles:
\begin{strategy}\label{strategy}
Fix a smooth curve $C$ and a line bundle $N$ on it. Let  $\psi: \tilde{C} \rightarrow C$ be a general spectral cover of degree $n$ with respect to $N$. In other words: The moduli space for these spectral covers is $\bigoplus_{j=1}^n H^0(N^j)$ and $\psi$ is a general element of it. We want to prove that $\tilde{C}$ does not admit a $g^r_d$. By the Beauville-Narasimhan-Ramanan correspondence we have set-theoretically
\begin{equation*}
	W^r_d(\tilde{C}) = h^{-1}(\tilde{C}) \cap W^r_C(N)(n,d).
\end{equation*}
Therefore $W^r_d(\tilde{C})$ is empty if and only if $h$ restricted to $W^r_C(N)(n,d)$ is not dominant.  If instead there exists a $g^r_d$ on $\tilde{C}$ we have 
\begin{equation*}
	\dim W^r_d(\tilde{C})  = \dim W^r_C(N)(n,d) - \dim \bigoplus_{j=1}^n H^0(N^j).
\end{equation*}
To analyze $W^r_C(N)(n,d)$ we construct finitely many families of Higgs bundles $(\cV_{T_i},\phi_{T_i})$ over base schemes $T_i$ with the property that $h^0(\cV_t) \geq r+1$ for all geometric point $t \in T_i$. We want the induced maps $T_i  \rightarrow W^r_C(N)(n,d)$ to be jointly surjective.

We sketch how we generally build the family of Higgs bundles. First we start with a suitable family of vector bundles $\cV_{S_i}$ on $C$ paramatrised by a scheme $S_i$. Then we will put 
\begin{equation*}
T_i = \left\{ (V,\phi): \text{$V$ is a vector bundle in $S_i$ and $\phi$ is a map $V \to V\otimes N$} \right\}.
\end{equation*} 
In other words $T_i$ is the union of $H^0(\End(V)\otimes N)$ as $V$ varies over the points of $S_i$. This is a special case of a relative section space. For the lack of a good reference in the literature we will construct it in the next subsection.
\end{strategy}
\subsection{Relative section space}\label{subsection:rel_section}
Consider $\nu: X\to Y$ a flat projective morphism of $\Bbbk$-schemes whose fibers are smooth projective curves of genus $g$ and let $\cL$ be a coherent sheaf on $X$ which is flat over $Y$. We will need a scheme which is glued from the spaces $H^0(X_y,\cL_y)$ as $y$ varies over the points of $Y$. We use the language of functors to give a precise definition of such a scheme.

\begin{definition}
We define the functor  $S_{\nu,\cL}: \op{Sch}/Y \to \op{Set}$:
\begin{equation}
S_{\nu,\cL}: T \mapsto H^0(X_T,\cL_T).
\end{equation}
Here $\cL_T$ and $X_T\usetikzlibrary{calc,babel}$ denote base change along the map $T \to Y$.
We call this the functor of \textbf{relative section space}.
\end{definition}
First we show that the functor is representable. We will call the corresponding scheme $S_{\nu,\cL}$ as well. 
\begin{proposition}\label{prop:relative_section}
Take $\cL$ to be a locally free sheaf on $X$ flat over $Y$. Furthermore assume there is a Cartier divisor $\Gamma$ in $X$ flat over $Y$ which has positive relative degree.
Then the functor $S_{\nu,\cL}$ is representable.
\end{proposition}
\begin{proof} (due to Hannah Larson and Feiyang Lin)
We first take $\cF$ to be a vector bundle on $X$ flat over $Y$ and assume $H^1(X_y,\cF_{y}) = 0$ for all closed points $y \in Y$ . We show that $S_{\nu,\cF}$ is represented by $\Aff(\nu_*\cF)$. Take $(f:T\to Y) \in \op{Sch}/Y$. From cohomology and base change we deduce 
\begin{equation}
f^*\nu_*\cF \cong \nu_{T*}\cF_T.
\end{equation} 
Furthermore $\nu_*\cF$ is a vector bundle. 
By the universal property of the symmetric algebra $\Aff(\nu_*\cF)(T) \cong \op{Hom}_{\OO_Y}((\nu_*\cF)^\vee,f_*\OO_T)$. Furthermore 
\begin{align*}
\op{Hom}_{\OO_Y}((\nu_*\cF)^\vee,f_*\OO_T) &\cong \op{Hom}_{\OO_T}(f^*(\nu_*\cF)^\vee,\OO_T)\\ &\cong \op{Hom}_{\OO_T}((f^*\nu_*\cF)^\vee,\OO_T) \\
&\cong H^0(f^*\nu_*\cF) \\
&\cong H^0(\nu_{T*}\cF_T).
\end{align*}
Here we used in the second line that pullback and dualizing commute for vector bundles $\nu_*\cF$. Furthermore we needed that vector bundles are reflexive.
Now we prove the original claim. Consider the sequence 
\begin{equation*}
0 \to \cL \to \cL(\Gamma) \to \cL(\Gamma)/\cL \to 0.
\end{equation*} 
It is exact because $\cL$ is locally free.
Take any $T \in \op{Sch}/Y$. Pull the last sequence back to $X_T$. By flatness of $\cL(\Gamma)/\cL$ it is still exact. Taking global section we find this exact sequence:
\begin{equation}\label{eq:kernel}
0 \to S_{\nu,\cL}(T) \to H^0(X_T,\cL(\Gamma)_T) \to H^0(X_T,\left(\cL(\Gamma)/\cL\right)_T).
\end{equation}
Let $d$ be the relative degree of $\cL$. We can replace $\Gamma$ by a multiple of it and assume that it has degree $> 2g-2-d$. For every closed point $y \in Y$ we have by Riemann-Roch $H^1(X_y,\cL(\Gamma)_y) = 0$ and $H^1(X_y,(\cL(\Gamma)_{|\Gamma})_y) = 0$. By the first part of the proof we know that $S_{\nu,\cL(\Gamma)}$ is represented by $\Aff(\nu_*\cL(\Gamma))$ and $S_{\nu,\cL(\Gamma)_{|\Gamma}}$ is represented by $\Aff(\nu_*\cL(\Gamma))_{|\Gamma})$. Now consider the zero section $\sigma: Y \to \Aff(\nu_*\cL(\Gamma))_{\Gamma})$ and the following fiber product:
\begin{equation}
\begin{tikzcd}
	T \\
	& {\mathbb{A}(\nu_*\mathcal{L}(\Gamma))\times_{\mathbb{A(\nu_*}\mathcal{L}(\Gamma)/\mathcal{L})}Y} & Y \\
	& {\mathbb{A}(\nu_*\mathcal{L}(\Gamma))} & {\mathbb{A}(\nu_*\mathcal{L}(\Gamma)/\mathcal{L})}
	\arrow[from=1-1, to=2-3]
	\arrow[curve={height=18pt}, from=1-1, to=3-2]
	\arrow["{\operatorname{pr}_1}"', from=2-2, to=2-3]
	\arrow["{\operatorname{pr}_2}", from=2-2, to=3-2]
	\arrow["\sigma", from=2-3, to=3-3]
	\arrow[from=3-2, to=3-3]
\end{tikzcd}
\end{equation}
The universal property of fiber products and equation (\ref{eq:kernel}) give us an isomorphism
\begin{equation*}
\left(\mathbb{A}(\nu_*\mathcal{L}(\Gamma))\times_{\mathbb{A}(\nu_*\mathcal{L}(\Gamma)_{|\Gamma})}Y \right)(T) \cong S_{\nu,\cL}(T).
\end{equation*}
Hence this fiber product is the scheme we searched for.
\end{proof}
Now we assume that $Y$ is irreducible. To compute the dimension of a relative section space we use the following degeneracy loci: 
\begin{equation*}
Y_p := \left\{ y \in Y: h^0(X_y,\cL_y) \leq r_{\op{gen}} + p \right\}
\end{equation*}
Here $r_{gen}$ is the rank of $h^0(X_y,\cL_y)$ at a general point of $Y$. We give these loci a natural scheme structure: Take $\Gamma$ as in the proof of Proposition \ref{prop:relative_section}. Let $\phi$ be the map $\nu_*\cL(\Gamma) \to \nu_*\cL(\Gamma)/\cL$. Then we have $Y_p = \left\{ y \in Y: \op{rk} \phi_{|y} \leq \op{rk}\nu_*\cL(\Gamma) - r_{\op{gen}} - p \right\}$. Hence $Y_p$ is a degeneracy loci and they have a natural scheme structure, see for example \cite[Chapter II]{ACGH}.
\begin{lemma}\label{lemma:dimension_section_space}
The dimension of $S_{\nu,\cL}$ is the following:
\begin{equation*}
\dim S_{\nu,\cL} = \op{max}_p \left(\dim Y_p + r_{\op{gen}}+p\right).
\end{equation*}
In particular if $\op{codim} Y_p \geq p $ for all $p \geq 1$ the dimension is $\dim Y + r_{\op{gen}}$. 
\end{lemma}
\begin{proof}
Let $\pi$ be the map from $S_{\nu,\cL}$ to $Y$. 
The relative section space can be written as a disjoint union of locally closed subsets like this:
\begin{equation*}
S_{\nu,\cL} = \bigsqcup_{p=0}^\infty \pi^{-1}\left( Y_p\setminus Y_{p+1}\right).
\end{equation*}
We note that $\pi^{-1}\left( Y_p\setminus Y_{p+1}\right)$
is the relative section space of $\cL_{Y_p\setminus Y_{p+1}}$ which in turn is isomorphic to $\Aff\left(\nu_{Y_p\setminus Y_{p+1}*}\cL_{Y_p\setminus Y_{p+1}}\right)$. This is a geometric vector bundle of rank $r_{\op{gen}} + p$ and has dimension $\dim Y_p + r_{\op{gen}}+p$. 
\end{proof}

\section{Spectral curves with a rational base}\label{section:rat_base}
This section was inspired by Beauville's article \cite{Beauville2}.
Here we concern ourselves with smooth spectral curves $\tilde{C}$ over $\PP^1$ with respect to the line bundle $\OO(k)$. We denote the associated branched cover by $\psi: \tilde{C} \rightarrow \PP^1$. We fix $\deg \psi = n$. We extend Beauville's approach to all vector bundles on $\PP^1$. We want to works towards understanding the irreducible components and the dimension of the Brill-Noether loci $W^r_d(\tilde{C})$.  We need the following stratification of the Brill-Noether loci: 
\begin{definition}
A \textbf{splitting type} of rank $n$ consists of integers $e_1 < e_2 < \ldots < e_l$ and positive integers $m_1,\ldots,m_l$ such that $\sum_{j=1}^l m_j = n$ holds. We denote it by $(e,m)$. Define the splitting locus corresponding to $(e,m)$ by 
\begin{equation}
U^{(e,m)}(\tilde{C}) := 	\left\{ L \in \op{Pic}(\tilde{C}) : \psi_*L \cong \bigoplus_{j=1}^l \OO(e_j)^{\oplus m_j} \right\}.
\end{equation} 
\end{definition}
\begin{remark}
In  \cite{LLV} a different but equivalent definition of splitting types of rank $n$ is used. They consider sequences of integers $\vec{e} = (\tilde{e}_1 \leq \tilde{e}_2 \leq \ldots \leq \tilde{e}_n)$. These objects are in natural bijection with our splitting types.  
Their notion of a splitting locus $W^{\vec{e}}(\tilde{C})$ differs from the one given above: They take all line bundles having pushforward of type $\vec{e}$ or a specialization of it. In the case of a general branched cover of degree $n$ this corresponds to the closure of our splitting locus. 
\end{remark}
The relation between splitting loci and the classical Brill-Noether loci is the following: Write $E = \bigoplus_{j=1}^l \OO(e_j)^{\oplus m_j}$. Then
\begin{equation*}
U^{(e,m)}(\tilde{C}) \subset W^{h^0(E) -1}_{\sum_i e_im_i + k\binom{n}{2}}(\tilde{C}).
\end{equation*} 
holds as a consequence of equation (\ref{eq:push_and_det}).

We view line bundles $L$ on $\tilde{C}$ as Higgs bundle on $\PP^1$: The pushforward $\psi_*L$ comes with a twisted endomorphism $\phi: \psi_*L \rightarrow \psi_*L \otimes \OO_{\PP^1}(k)$ which we view as an $n\times n$-matrix whose entries are sections of line bundles on $\PP^1$. 
\begin{lemma}\label{lemma:aut_vectorbundle}
Let $(e,m)$ be a splitting type and $V=\bigoplus_{j=1}^l \OO_{\PP^1}(e_j)^{\oplus m_j}$. Then the automorphism group of $V$ consists of the following lower triangular block matrices:
\begin{equation*}
\begin{pmatrix}
A_{11} & 0 & \cdots & 0 \\
B_{21} & A_{22} & \cdots & 0 \\
\vdots & \vdots & \ddots &\vdots \\
B_{l1} & B_{l2} &\cdots& A_{ll}
\end{pmatrix}
\end{equation*}
Here the matrices $A_{jj}$ are elements of $GL_{m_j}(\Bbbk)$ and the $B_{ij}$ are arbitrary matrices with entries in $H^0(\OO_{\PP^1}(e_i - e_j)^{\oplus m_im_j})$. The dimension of $\op{Aut}(V)$ as an algebraic group is then
\begin{equation*}
\dim\op{Aut}(V) = \sum_{1\leq i \leq j \leq l}(e_j - e_i+1)m_im_j.
\end{equation*}
\end{lemma}
\begin{proof}
One writes an endomorphism of $V$ as a block matrix $\left(\phi_{ij}\right)_{i,j=1}^l$ with $\phi_{ij} \in H^0(\OO(e_i)^{\oplus m_i} \otimes \OO(-e_j)^{\oplus m_j})$. If $i < j$ $\phi_{ij}$ is zero. The determinant of such a matrix is $\prod_j \det A_{jj}$ and needs to be nonzero. Hence $A_{jj}$ is in $GL_{m_j}(\Bbbk)$. 
\end{proof} 
We define a Higgs bundle version of the splitting loci. 
\begin{align*}
U^{(e,m)}(k) =  \left\{ (V,\phi: V \rightarrow V(k)) : \begin{aligned}
&\text{$V\cong \bigoplus_{j=1}^l \OO_{\PP^1}(e_j)^{\oplus m_j}$}\\
&\text{and $V$ does not have a $\phi$-invariant subbundle}
\end{aligned} \right\} / \sim.
\end{align*}
Here the equivalence relation $\sim$ we divide out is equivalence of Higgs bundles. This is a scheme because it is a locally closed subset of the moduli space of stable Higgs-bundles $\cH_{\PP^1}(\OO(k))(n,d)$. We construct this moduli space as a group quotient. For a fixed $(e,m )$ abbreviate the vector bundle $\bigoplus_{j=1}^l \OO(e_j)^{\oplus m_j}$ by $V$. Define the following Zariski-open subset of the space of its space of twisted endomorphisms:
\begin{equation*}
\tilde{U}^{(e,m)} (k)= \left\{ \phi \in H^0(\End(V) \otimes \OO(k)) : \text{there is no $\phi$-invariant subbundle}\right\}.
\end{equation*} 
We denote by $\op{Aut}_1(V)$ the automorphism group of $V$ modulo scalar multiples of the identity. This subgroup acts on $\tilde{U}^{(e,m)}$: $g \cdot \phi := g\phi g^{-1}$. 
Here $g$ is in $\op{Aut}_1(V)$ and $\phi $ in $\tilde{U}^{(e,m)}(k)$. All elements of $U^{(e,m)}(k)$ are stable Higgs bundles. Therefore the action is free by Lemma \ref{lemma:aut}. Then we have
\begin{equation*}
U^{(e,m)}(k) \cong \tilde{U}^{(e,m)}(k)/\op{Aut}_1(V). 
\end{equation*}
We have the following analog of the Hitchin fibration 
\begin{align*}
h: U^{(e,m)}(k) &\rightarrow \bigoplus_{i=1}^d H^0(\OO_{\PP^1}(ik))\\
(V,\phi) &\mapsto \chi_{\phi}.
\end{align*}
We are now prepared to show our first result.
\begin{theorem}(Theorem A)
Let $(s_1,\ldots,s_n)$ be a general element of $\bigoplus_{i=1}^n H^0(\OO_{\PP^1}(ik))$ and denote by $\tilde{C}$ the associated spectral curve. Write $V$ for the vector bundle $\bigoplus_{j=1}^l \OO_{\PP^1}(e_j)^{\oplus m_j}$.
Then 
\begin{equation}\label{eq:dim_bn_loci}
	\dim U^{(e,m)}(\tilde{C}) \leq h^0(\End(V) \otimes \OO(k)) + 1-  \sum_{1\leq i \leq j \leq l}(e_j - e_i+1)m_im_j - d - k\binom{n+1}{2}
\end{equation}
holds for a splitting type $(e,m)$ if the right hand side is nonnegative. This is either an equality or the splitting locus $U^{(e,m)}(\tilde{C})$ is empty.
Take $(e,m)$ to be a splitting type satisfying $e_1 - e_l + k\geq -1$.   We have
\begin{equation}\label{eq:dim_split_loci}
	\dim U^{(e,m)}(\tilde{C}) = g(\tilde{C}) - \sum_{1 \leq i < j\leq l}(e_j - e_i -1 )m_im_j 
\end{equation}
if the right hand side is nonnegative.
Furthermore in this case the space $U^{(e,m)}(\tilde{C})$ is smooth.
\end{theorem}
\begin{proof} Assume that the right hand side in (\ref{eq:dim_bn_loci}) is nonnegative.
The Hitchin correspondence \ref{prop:hitchin_corr} tells us $h^{-1}(f) \cong U^{(e,m)}(\tilde{C})$.
Hence if we have a general element $f$ of the Hitchin base the dimension of its fiber can be bounded as follows:
\begin{align}\label{eq:dim_ineq}
\dim U^{(e,m)}(\tilde{C}) &\leq h^0(\End(V) \otimes \OO(k) ) - \dim \op{Aut}_1(V) - \sum_{i=1}^{n}h^0(\OO_{\PP^1}(ik)).
\end{align}
This is an equality if the Hitchin map is dominant. Otherwise the general fiber $h^{-1}(f)$ is empty. This shows the first assertion.
To prove the second one we use the assumption on the splitting data to find 
\begin{align*}
h^0(\End(V) \otimes \OO(k)) =& \sum_{j=1}^lh^0(\OO(k))m_j^2 \\&+ \sum_{1\leq i < j \leq l}\left( h^0(\OO(k + e_i - e_j)) + h^0(\OO(k+e_j - e_i))\right)m_im_j\\
=&  \sum_{j=1}^l (k+1)m_j^2+  \sum_{1 \leq i < j \leq l}2(k+1)m_im_j.
\end{align*}
If $(e,m)$ does not fulfill the assumption $h^0(\End(V) \otimes \OO(k))$ is instead strictly greater than this. 
Plugging in the previous equation in (\ref{eq:dim_ineq}) gives
\begin{align*}
\dim U^{(e,m)}(\tilde{C}) &\leq  \sum_{j=1}^l km_j^2+  \left(\sum_{1 \leq i < j \leq l}(2k+1 + e_i - e_j)m_im_j \right) + 1 - n- k\binom{n+1}{2} \\
&= 
k\binom{n}{2} + 1 - n + \sum_{1\leq i < j\leq l}(e_i - e_j +1)m_im_j \\
&= g(\tilde{C}) - \sum_{1 \leq i < j\leq l}(e_j - e_i -1 )m_im_j.
\end{align*} 
Lastly we show smoothness. The tangent space of an equivalence class $[\phi] \in U^{(e,m)}(k)$ is 
\begin{equation*}
T_{[\phi]}\left(\tilde{U}^{(e,m)}/\op{Aut}_1(V)\right) 
\cong T_\phi \tilde{U}^{(e,m)}/ \{ (\rho\otimes \op{id}_{\OO(k)})\circ\phi - \phi\circ\rho : \rho \in \op{End}(V)\}. 
\end{equation*}
By Lemma \ref{lemma:aut} the vector space $\{ (\rho\otimes \op{id}_{\OO(k)})\circ\phi - \phi\circ\rho : \rho \in \op{End}(V)\}$ has the same dimension as $\op{Aut}_1(V)$. Hence the tangent spaces of $U^{(e,m)}(k)$ have the same dimension as $U^{(e,m)}(k)$ and this space is smooth. We now apply generic smoothness \cite[Theorem 21.6.6]{Vakil} to find an open subset $V$ of the Hitchin base such that $h_{|h^{-1}(V)}$ is smooth. In particular if $\tilde{C}$ is the spectral curve of any $f \in V$ the space $U^{(e,m)}(\tilde{C})$ is smooth.
\end{proof}
\begin{remark}\label{remark:larson comparison}
Let us compare this to Theorem \ref{thm:hurwitz_brill_noether} in the case where the splitting locus satisfies $e_i - e_j + k\geq -1$ for all $i < j$. They denote their splitting types by $\vec{e} = (\tilde{e}_1 \leq \ldots \leq \tilde{e}_n)$ i.e $V \cong \oplus_{i=1}^d \OO(\tilde{e}_i)$. They introduce the expected dimension
\begin{equation*}
\rho'(\vec{e}) = g - \sum_{1\leq i < j\leq n}\op{max}\{0,\tilde{e}_j - \tilde{e}_i -1\}.
\end{equation*} 
Note that we only need to sum in this equation over $i < j$. Writing $\vec{e}$ as $(e_1,\ldots,e_l,m_1,\ldots,m_l)$ the summand $e_r - e_s - 1$ for $1 \leq r < s \leq l$ appears $m_im_j$ many times in the above equation. Therefore this agrees with our equation (\ref{eq:dim_split_loci}).

In the preprint \cite{larsonvemulapalli} the same result was obtained independently in the more general context of curves on Hirzebruch surfaces by Larson and Vemulapalli. Furthermore they determine all the splitting types for which $U^{(e,m)}(\tilde{C})$ is non-empty.
\end{remark}
\begin{remark}
A general degree $n$ spectral cover of $\PP^1$ with respect to $\OO(k)$ is simply branched. To see this look at the closed subvariety $X$ of $B:=\bigoplus_{i=0}^n H^0(\OO(ik))$ consisting of all spectral covers $\tilde{C} \rightarrow \PP^1$ which are not simply branched. It has three irreducible components $X_1, X_2$ and $X_3$. The first consists of all singular spectral curves.  If the source $\tilde{C}$ of the spectral cover  $\psi$ is instead smooth we need to have a branch point $p \in \PP^1$ such that the branch divisor has multiplicity $3$ at $p$. This implies that $p$ either has a preimage $q$ having ramification order $3$ or two preimages both of ramification order $2$. In analogy with plane curves we call these possibilites flex and bitangent respectively. The closure of the locus of spectral curves having a flex resp. a bitangent will be the affine varieties $X_2$ resp. $X_3$. We sketch how to prove that $X_2$ is irreducible of positive codimension. The argument for the other two is similar. Look at the incidence correspondence 
\begin{equation*}
\tilde{X}_2 = \left\{ ((s_1,\ldots,s_n),p) \in B \times \PP^1 : \text{The spectral curve has a flex at $p$} \right\}.
\end{equation*}
The image of $\tilde{X}_2$ under projection on the first coordinate is $X_2$ so it suffices to find the dimension and the irreducible components of $X_2$. Consider the projection on the second coordinate $\pi_2: \tilde{X}_2 \rightarrow \PP^1$. Consider the affine variety $\op{Trip}_n \subset \Aff^n$ of monic polynomials of degree $n$ having a triple root. Up to the choice of a local parameter we have an evaluation function at $p \in \PP^1$:
\begin{align*}
\op{ev}_p: \pi_2^{-1}(p) &\to \op{Trip}_n
\\
((s_1,\ldots,s_n),p) &\mapsto s_1(p),\ldots,s_n(p).
\end{align*}
The target space is irreducible of dimension $n-2$. The evaluation map is surjective and because it is linear the fibers are of dimension $\dim B - n \geq 0$. Hence $\pi_2^{-1}(p)$ is irreducible of dimension $\dim B -2$ and our claim about $\tilde{X}_2$ follows.
\end{remark}
\section{The restricted Hitchin map}
\label{section:restricted_hitchin}
Throughout this chapter $C$ is a smooth projective curve of genus $g \geq 2$ over $\Bbbk$.
\begin{definition}
Let $V$ be a vector bundle of rank $2$ on $C$ and $\Gamma$ an effective divisor on $C$. We denote the sheaf of tracefree endomorphisms of $V$ by $\End_0(V)$. Then we define the \textbf{meromorphic restricted Hitchin map} to be 
\begin{equation*}
h_{V,\Gamma}: H^0(\End_0(V)\otimes K_C(\Gamma)) \to H^0(K_C^2(2\Gamma)),\quad\phi\mapsto \det \phi.
\end{equation*} 
If $\Gamma = 0$ we simply speak of the \textbf{restricted Hitchin map} and write $h_{V,\Gamma} = h_V$.
\end{definition}
We introduce matrix notation in case $V$ is a direct sum of line bundles:
\begin{notation}
Let $V=M \oplus N$ for line bundles $M,N$. We write any $\phi \in H^0(\End_0(V)\otimes K_C)$ as $\begin{pmatrix}
\phi_{11} & \phi_{12}\\
\phi_{21} & -\phi_{11}
\end{pmatrix}$ where $\phi_{11}\in H^0(K_C),\phi_{12}\in H^0(M\otimes N^\vee\otimes K_C), \phi_{21} \in H^0(M^\vee\otimes N\otimes K_C)$. Then we have $h_V(\phi) = -\phi_{11}^2 - \phi_{12}\phi_{21}$.
\end{notation}
We fix a vector bundle $V$. We are interested if there is a twisted endomorphism of it with a given general spectral curve. The next lemma reduces this question to the study of the restricted Hitchin map.
\begin{lemma}
Let $V$ be a rank $2$ vector bundle on $C$. Assume that the restricted Hitchin map $h_V$ is dominant. Then also the following map has dense image:
\begin{equation*}
H_V: H^0(\End(V)\otimes K_C) \to H^0(K_C) \oplus H^0(K_C^2), \phi \mapsto (\op{tr} \phi,\det\phi).
\end{equation*} 
\end{lemma}
\begin{proof}
Take a general element $(s_0,s_1)$ of $H^0(K_C)\oplus H^0(K_C^2)$. Then $s_2 - \frac{1}{4}s_1^2$ is general in $H^0(K_C^2)$ and by assumption there exists $\phi_0 \in H^0(\End_0(V)\otimes K_C)$ such that $h_V(\phi) = s_2 - \frac{1}{4}s_1^2$. Then $\phi_0 + \frac{1}{2}\op{id}_V$ has trace $s_1$ and determinant $s_2$.
\end{proof}
\begin{lemma}\label{lemma:bpf}
Take $C$ to be a smooth projective curve of genus $g \geq 2$ which is not hyperelliptic. Assume $V$ is a vector bundle of rank $2$ and degree $d$ which does not contain line-subbundles of degree $\geq \frac{d}{2}-1$. Then $\End_0(V)\otimes K_C$ is basepoint-free.
\end{lemma}
\begin{proof}
Let $q$ be any point of $C$. Consider the following exact sequence:
\begin{align*}
0 \to \End_0(V)\otimes K_C(-q) \to \End_0(V)\otimes K_C\to \End_0(V)\otimes K_{C|q} \to 0.
\end{align*}
From its long exact sequence we see that $H^0(\End_0(V)\otimes K_C)\to H^0(\End_0(V)\otimes K_{C|q})$ is surjective if $H^1(\End_0(V)\otimes K_C(-q)) = 0$. By Serre duality we have 
\begin{equation*}
H^1(\End_0(V)\otimes K_C(-q))^\vee \cong H^0(\End_0(V)\otimes \OO_C(q)).
\end{equation*}
We show that the right hand side is zero. Assume for contradiction there is a nonzero twisted endomorphism $\rho: V \to V(q)$. Then the characteristic polynomial $\chi_\rho$ is in $H^0(\OO_C(q))\oplus H^0(\OO_C(2q)) \cong k^{\oplus 2}$. Here we used that $C$ is not hyperelliptic. We can factor $\chi_\rho$ as $\chi_\rho(t) = (t-\lambda_1)(t-\lambda_2)$ with $\lambda_1,\lambda_2 \in \Bbbk$. Then consider $\rho_i = \rho - \lambda_i \op{id}'_V$. Here we use the natural injective map $\OO_C \to \OO_C(q)$ to define $\op{id}'_V: V \to V(q)$. Then we consider the image of $\rho_i$. It is a line bundle because $\chi_\rho(\lambda_i) = 0$. Hence $\ker(\rho_i)$ is a subbundle of $V$. We have $\deg \op{im}(\rho_i) \leq \mu(V(p))-1 = \frac{d}{2}$. On the other hand $d = \deg \ker(\rho_i)+\deg \op{im}(\rho_i)$. This implies $\deg \ker(\rho_i) \geq \frac{d}{2}$. This contradicts our assumption on the degree of subbundles.
\end{proof}
\subsubsection{Quadrics of rank at most $3$ containing a curve}
In the proofs of our next result we need some facts about quadrics of low rank containg a curve. We use the work of \cite{kadiklöylü}. First we define the variety in question: Let $L$ be a line bundle on $C$. We denote by $Q_k(C,L)$ the projective variety in $\PP\op{Sym}^2H^0(C,L)^\vee$ consisting of quadrics of rank $\leq k$ containing the image of $C$ in $\PP H^0(C,L)$. We write $\hat{Q}_k(C,L)$ for the cone of $Q_k(C,L)$ in $\op{Sym}^2H^0(C,L)$. Now $Q_k(C,L)$ has the following expected dimension:
\begin{equation*}
q(g,r,d,k) := \binom{r+2}{2}-\binom{r-k+2}{2}-2d+g-2.
\end{equation*}    
In \cite{kadiklöylü} it is proved that the expected dimension is under some assumptions the true dimension:
\begin{lemma}\label{lemma:quadrics}
Let $C$ be a general curve of genus $g$ and let $L$ be a general element of $W^r_d(C)$ such that $g-d+r\geq 1$ holds. Then the variety $Q_k(C,L)$ is of pure dimension $q(g,r,d,k)$.
\end{lemma}
\begin{proof}
See \cite[Theorem 1.1]{kadiklöylü}.
\end{proof}
\begin{lemma}\label{lemma:quadrics,canonical}
Let $C$ be a general curve of genus $g$. Then $Q_k(C,K_C)$ is of pure dimension ${q(g,g-1,2g-2,k)}$
\end{lemma}
\begin{proof}
See \cite[Lemma 2.1]{kadiklöylü}.
\end{proof}

\subsubsection{The restricted Hitchin map of the trivial vector bundle of rank $2$}

The rest of this section is devoted to the question: For which vector bundles is $h_{V,\Gamma}$ dominant? We first answer this question partially for $\OO_C^{\oplus 2}$. We thank Daniele Agostini for suggesting the idea behind the proof of Proposition \ref{prop:trivial_bundle_mer}. 

We explain how quadrics come into play: We have the following factorization of $h_{V,\Gamma}$:
 \begin{equation}\label{eq:factorization}
H^0(\End_0(\OO_C^{\oplus 2})\otimes K_C(\Gamma)) \xrightarrow{\delta} \op{Sym}^2H^0(K_C(\Gamma)) \xrightarrow{\op{ev}} H^0(K_C^2(2\Gamma)).
\end{equation}
The first map sends $\phi$ to $-\phi_{11}^2 - \phi_{12}\phi_{21}$, where the multiplication happens in the symmetric algebra. The second map is the usual multiplication map of sections. The image of $\delta$ is the locus of quadrics of rank at most $3$ denoted by $\hat{Q}_3$: Every image is clearly a quadric of rank $\leq 3$. On the other hand take a quadric of the form $q = \eta_1^2 + \eta_2^2 + \eta_3^2$ with the $\eta_i \in H^0(K_C)$. We can rewrite this:
\begin{equation*}
q = \eta_1^2 + (\eta_2 + \sqrt{-1}\eta_3)(\eta_2^2 - \sqrt{-1}\eta_3).
\end{equation*}   
Now it is clear that $q$ is in the image. We prove that $\hat{Q}_3$ has dimension $h^0(\End_0(\OO_C^{\oplus 2})\otimes K_C(\Gamma)) - 3$. Consider the open set $U$ of $\hat{Q}_3$ consisting of quadrics of rank exactly $3$. Then $\delta^{-1}(U)$ is a principal $PGL_2$-bundle over $U$ and $\delta^{-1}(U)$ consists of all $\phi$ such that $\phi_{11},\phi_{12}, \phi_{21}$ are linearly independent. With this preparation we can prove the following Proposition.
\begin{proposition}\label{prop:trivial_bundle_mer}
Let $C$ be a general curve of genus $g\geq 3$ and let $\Gamma$ be either $0$ or a general effective divisor of degree $\geq 3$. Then the map $h_{\OO_C^{\oplus 2},\Gamma}$ is dominant.  
\end{proposition} 
\begin{proof}
We use the factorisation \begin{equation*}
H^0(\End_0(\OO_C^{\oplus 2})\otimes K_C(\Gamma)) \xrightarrow{\delta} \op{Sym}^2H^0(K_C(\Gamma)) \xrightarrow{\op{ev}} H^0(K_C^2(2\Gamma)).
\end{equation*}
We have seen that the image of $\delta$ has dimension $h^0(\End_0(\OO_C^{\oplus 2})\otimes K_C(\Gamma)) - 3 = 3g + 3\deg \Gamma - 6$. On the other hand $H^0(K_C^2(2\Gamma))$ has dimension $3g -3 + 2\deg \Gamma$. Hence by semicontinuity if we show that $\op{ev}_{|\op{im}\delta}^{-1}(0)$ has dimension $\deg \Gamma -3$ we are done. This is $\hat{Q}_3(C,K_C(\Gamma))$. If $\deg \Gamma \geq 3$ we apply Lemma \ref{lemma:quadrics}. Our line bundle $K_C(\Gamma)$ has $d = 2g- 2 + \deg \Gamma$ and $r = g + \deg \Gamma -2$. Note that $g- d+ r = 0$ so we can apply the Theorem cited above. We find 
\begin{align*}
\dim \op{ev}_{|\op{im}\delta}^{-1}(0) &= q(g,r,d,k) +1 \\&= \binom{g+\deg \Gamma}{2} - \binom{g+\deg \Gamma -3}{2} - 2(2g-2 + \deg \Gamma) + g -1\\
& = \deg \Gamma -3.
\end{align*}
This is the number we wanted. If $\Gamma$ is trivial we instead use Lemma \ref{lemma:quadrics,canonical} to show that $\op{ev}^{-1}_{|\op{im}\delta}(0)$ is zero-dimensional.
\end{proof}
\begin{remark}
The above Proposition fails for hyperelliptic curves because the multiplication map $H^0(K_C) \otimes H^0(K_C) \to H^0(K_C^2)$ is not surjective.
\end{remark}
\subsection{The twisted endomorphisms of an extension}
We start with a simplifying observation:
\begin{remark}
If $h_V$ is dominant, $h_{V\otimes N}$ is also dominant for every line bundle $N$ on $C$. This follows from the following commutative triangle in which the vertical map is a isomorphism:
\begin{equation*}
\begin{tikzcd}
	{H^0(\End_0(V)\otimes K_C)} & {H^0(K_C^2)} \\
	{H^0(\End_0(V\otimes N)\otimes K_C)}
	\arrow["{h_V}", from=1-1, to=1-2]
	\arrow["{(-)\otimes \op{id}_N}"', from=1-1, to=2-1]
	\arrow["{h_{V\otimes N}}"', from=2-1, to=1-2]
\end{tikzcd}
\end{equation*}
\end{remark}
By the preceding remark we may assume for simplicity that our vector bundle $V$ sits in an exact sequence like this:
\begin{equation}
0 \to L \to V \to \OO_C \to 0.
\end{equation} 
These extensions are classified by extension classes $ \xi \in \op{Ext}^1(\OO_C,L) \cong H^1(L)$. To analyze the twisted endomorphism of $V$ we use an explicit construction of $V$ from a given $\xi$. Take an effective divisor $\Gamma$ on $C$. Consider the exact sequence:
\begin{equation*}
0 \to L \to L(\Gamma) \to L(\Gamma)_{|\Gamma} \to 0..
\end{equation*}
Its long exact sequence gives us a map $H^0(L(\Gamma)_{|\Gamma})\to H^1(L) \cong \op{Ext}^1(\OO_C,L)$. We consider the extension class $\xi$ induced by $\tilde{\xi} \in H^0(L(\Gamma)_{|\Gamma})$. We regard $\tilde{\xi}$ as a morphism $\OO_C \to L(\Gamma)_{|\Gamma}$.
\begin{lemma}\label{lemma:yoneda}
The extension corresponding to the above $\xi$ is given by the vector bundle 
\begin{equation}\label{eq:ext_as_a_subbundle}
V = \op{ker}\left(L(\Gamma) \oplus \OO_C \xrightarrow{(\op{ev}_\Gamma,-\tilde{\xi})} L(\Gamma)_{|\Gamma} \right).
\end{equation}
The maps $L \to V$ and $V \to \OO_C$ are induced by the injection $L\to  L(\Gamma)\oplus \OO_C$ and the projection $L(\Gamma)\oplus \OO_C \to \OO_C$ respectively. 
\end{lemma}
\begin{proof}
Pick an injective resolution $I^\bullet$ of $L$. Then there exists a map between complexes 
\begin{equation*}
\begin{tikzcd}
	0 & L & {L(\Gamma)} & {L(\Gamma)_{|\Gamma}} & 0 \\
	0 & L & {I^0} & {I^1} & \ldots
	\arrow[from=1-1, to=1-2]
	\arrow[from=1-2, to=1-3]
	\arrow["{\operatorname{id}_L}", from=1-2, to=2-2]
	\arrow[from=1-3, to=1-4]
	\arrow["{f^0}", from=1-3, to=2-3]
	\arrow[from=1-4, to=1-5]
	\arrow["{f^1}", from=1-4, to=2-4]
	\arrow[from=2-1, to=2-2]
	\arrow[from=2-2, to=2-3]
	\arrow["{d^0}", from=2-3, to=2-4]
	\arrow[from=2-4, to=2-5]
\end{tikzcd}
\end{equation*}
Let $\xi'$ be the image of $\tilde{\xi}$ in $\op{Hom}(\OO_C,I^1)$. Then we have 
\begin{equation}\label{eq:isom_exts}
g: V \isomarrow \op{ker}\left( I^0 \oplus \OO_C \xrightarrow{(d^0,-\xi')} I^1\right)
\end{equation}
Here $g$ is the restriction of $(f^0,\op{id}_{\OO_C})$. It is injective because $f^0$ is injective: The kernels of $f^0$ and $f^1$ are isomorphic and $\ker(f^0)$ is torsion. Because $L(\Gamma)$ is locally free we deduce $\ker f^0 = 0$. The surjectivity of $g$ is a straightforward diagram chase. 
It is then well known from homological algebra that the right hand side in (\ref{eq:isom_exts}) is the extension corresponding to $\xi \in H^1(L)$. See e.g Chapter 3 of \cite{weibel}.
\end{proof}
We note that the twisted endomorphisms of $V$ are naturally a meromorphic twisted endomorphisms of $W$: If $\phi$ is in $H^0(\End_0(V)\otimes K_C)$ then 
\begin{equation*}
	W \to V(q) \xrightarrow{\phi\otimes \op{id}_{\OO(q)}} \to V\otimes K_C(q) \to W \otimes K_C(q)
\end{equation*}
is in $H^0(\End_0(W)\otimes K_C(q))$. Let us find $H^0(\End_0(V)\otimes K_C)$ inside $H^0(\End_0(W)\otimes K_C(q))$ explicitly.
To do this we will assume that $\Gamma$ has the form $q_1 + \ldots + q_m$, where $q_1,\ldots,q_m$ are pairwise distinct points of $C$. Abbreviate $L(\Gamma)\oplus \OO_C$ by $W$. We decompose $\tilde{\xi}$ into maps $\xi(q_j) : \OO_C \to L(q_j)_{|q_j}$.
For each $j$ we choose a local parameter $t$ at $q_j$ and an isomorphism $L(\Gamma)_{q_j}^\wedge \isomarrow \Bbbk[\![t]\!]$. Here $(-)^{\wedge}$ denotes the completion. We get an isomorphism $W^\wedge_{q_j} \isomarrow \Bbbk[\![t]\!]^{\oplus 2}$. The completion $V_{q_j}^\wedge$ is then the following:
\begin{equation*}
V^\wedge_{q_j} = \left\langle \begin{pmatrix}
t \\0
\end{pmatrix}, \begin{pmatrix}
\xi(q_j) \\ 1
\end{pmatrix} \right\rangle
\end{equation*}
We denote the standard basis of $\Bbbk[\![t]\!]^{\oplus 2}$ by $\mathcal{B}$ and the above basis of $V^\wedge_{q_j}$ by $\mathcal{B}'$. Assume we have a meromorphic twisted endomorphism $\phi$ of $W$. Write it in $V^\wedge_{g_j}$ as  ${\phi= \begin{pmatrix}
\phi_{11} & \phi_{12} \\
\phi_{21} & -\phi_{11}
\end{pmatrix}}$. The entries are Laurent series in $t$.
 We now express the twisted endomorphism in $\mathcal{B}'$:
\begin{align*}
M^\mathcal{B}_{\mathcal{B}'}\begin{pmatrix}
\phi_{11} & \phi_{12} \\
\phi_{21} & -\phi_{11}
\end{pmatrix}M^{\mathcal{B}'}_\mathcal{B} &= 
\begin{pmatrix}
t & \xi(q_j)\\
0 & 1
\end{pmatrix}^{-1}
\begin{pmatrix}
\phi_{11} & \phi_{12} \\
\phi_{21} & -\phi_{11}
\end{pmatrix}
\begin{pmatrix}
t & \xi(q_j) \\
0 & 1
\end{pmatrix} \\
&= \begin{pmatrix}
\phi_{11}-\phi_{21}\xi(q_j) & t^{-1}(2\phi_{11}\xi(q_j) +\phi_{12}-\phi_{21}\xi(q_j)^2) \\
t\phi_{21} & -\phi_{11}+\phi_{21}\xi(q_j)
\end{pmatrix}
\end{align*}
Because we work in the basis $\mathcal{B}'$, the entries of this matrix are in $\Bbbk[\![t]\!]$ if and only if $\phi$ is a twisted endomorphism of $V$. We express this differently by writing out the first terms of the Laurent expansion at the point $q_j$: $\phi_{ij} = \sum_l \phi_{ij}^{(l)}(q_j)t^l$. 
\begin{lemma}\label{lemma:key_eqs}
Let $V,W,\Gamma$ as before. Then $H^0(\End_0(V)\otimes K_C)$ is the subspace of $H^0(\End_0(W)\otimes K_C(\Gamma))$ cut out by the following linear equations:
\begin{align}\label{eq:coeff_vanishing_1}
0 &= \phi^{(-1)}_{11}(q_j)-\phi_{21}^{(-1)}(q_j)\xi(q_j), \\\label{eq:coeff_vanishing_2}
0 &= 2\phi^{(-1)}_{11}(q_j)\xi(q_j) + \phi^{(-1)}_{12}(q_j) - \phi^{(-1)}_{21}(q_j)\xi(q_j)^2 \\\label{eq:coeff_vanishing_3},
0 &= 2\phi^{(0)}_{11}(q_j)\xi(q_j) + \phi^{(0)}_{12}(q_j) - \phi_{21}^{(0)}(q_j)\xi(q_j)^2, \qquad j =1,\ldots,m.
\end{align}
\end{lemma} 
Here is a geometric consequence: 
\begin{corollary}\label{cor:nilpotent_fiber}
Let $V,W,\Gamma$ as before. Then a meromorphic twisted endomorphism $W \to W\otimes K_C(\Gamma)$ is in $H^0(\End_0(V)\otimes K_C(\Gamma))$ only if the morphism $W_{|q_j} \to \left(W\otimes  K_C(\Gamma)\right)_{|q_j}$ is nilpotent and $\begin{pmatrix}
\xi(q_j) & 1
\end{pmatrix}^T$ is contained in its kernel.
\end{corollary}
\begin{proof}
Equations (\ref{eq:coeff_vanishing_1}) and (\ref{eq:coeff_vanishing_2}) imply
\begin{equation*}
\phi^{(-1)}(q_j) = \phi_{21}^{(-1)}(q_j)\begin{pmatrix}
\xi(q_j) & -\xi(q_j)^2 \\
1 & -\xi(q_j)
\end{pmatrix}
\end{equation*}
The right hand side is nilpotent and $\begin{pmatrix}
\xi(q_j) \\ 1
\end{pmatrix}$ is in the kernel.  
\end{proof}
\subsection{Application to vector bundles with two sections}

We apply the technique developed in this section to stable vector bundles $V$ with $h^0(V)\geq 2$. We introduce the following Brill-Noether locus in $\mathcal{U}_C(2,d)^s$ - the moduli space of stable vector bundles of degree $d$:
\begin{equation*}
W^1_C(2,d) := \left\{ E \in\mathcal{U}_C(2,d)^s : h^0(E)\geq 2 \right\}.
\end{equation*}
\begin{lemma}\label{lemma:restricted_hitchin_3,4}
Let $C$ be a general curve of genus $g \geq 2$ and let $V$ be a general vector bundle in $W^1_C(2,d)$ where $d = 3,4$. Then $h_V$ is dominant.
\end{lemma}
\begin{proof}
In \cite{Teixidor} it is shown that a general stable vector bundle $V$ in $W^1_C(2,d)$ admits an extension
\begin{equation}\label{eq:ext_teixidor}
0 \to \OO_C \to V \to \OO_C(\Gamma) \to 0,
\end{equation}
where $\Gamma$ is an effective divisor of degree $d$ with $h^0(\OO_C(\Gamma)) = 1$. Let us write $\Gamma = p_1 + \ldots + p_d$. From the proof given in \cite{Teixidor} it follows that we can assume the points $p_i$ to be pairwise distinct. The extension gives us a point $[\xi] \in \PP H^0(K_C(\Gamma))$. By Lemma \ref{lemma:lift_section} a necessary and sufficient condition for $h^0(V)\geq 2$ is that $[\xi]$ lies in the span $\langle \Gamma \rangle \subset \PP H^0(K_C(\Gamma))$. We can assume that $[\xi]$ is a general point of $\langle \Gamma \rangle$ because stability is an open property. Let $V_0$ be $V\otimes \OO_C(-\Gamma)$. The Lemma will be proved if we show that $h_{V_0}$ is dominant. For $\phi \in H^0(\End_0(\OO^{\oplus 2})\otimes K_C(\Gamma))$ and $q\in C$ define 
\begin{equation*}
\op{res}_p(\phi) := \begin{pmatrix}
\op{res}_p\phi_{11} & \op{res}_p \phi_{12} \\
\op{res}_p \phi_{21} & -\op{res}_p \phi_{11}
\end{pmatrix}
\end{equation*}
Recall from Lemma \ref{lemma:key_eqs} that $H^0(\End_0(V)\otimes K_C) \subseteq H^0(\End_0(\OO_C^{\oplus 2} ) \otimes K_C(\Gamma))$. Define $Q$ to be $\delta(H^0(\End_0(V_0) \otimes K_C))$. From the factorization in (\ref{eq:factorization}) we obtain a commutative diagram:
\begin{equation}
\begin{tikzcd}
	{H^0(\mathcal{E}nd_0(\mathcal{O}^{\oplus 2})\otimes K_C(\Gamma))} & {\operatorname{Sym}^2H^0(K_C(\Gamma))} & {H^0(K_C^2(2\Gamma))} \\
	{H^0(\mathcal{E}nd_0(V_0)\otimes K_C)} & Q & {H^0(K_C^2)}
	\arrow["\delta", from=1-1, to=1-2]
	\arrow["{\operatorname{ev}}", from=1-2, to=1-3]
	\arrow[hook, from=2-1, to=1-1]
	\arrow[from=2-1, to=2-2]
	\arrow[hook, from=2-2, to=1-2]
	\arrow["{\operatorname{ev}_{|Q}}", from=2-2, to=2-3]
	\arrow[hook, from=2-3, to=1-3]
\end{tikzcd}
\end{equation}
We first show that $\delta_{|H^0(\End_0(V_0)\otimes K_C)}$ is generically injective. This proves that $\dim Q = 3g -3$. Consider $U$ the open subset of $H^0(\End_0(V_0)\otimes K_C)$ such that $\phi_{11},\phi_{12}$ and $\phi_{21}$ are linearly independent and $\op{res}_{p_i}\phi \neq 0$ holds for $i=1,\ldots,d$. Take any $\phi \in U$ and put $q  = \delta(\phi) $. We first show that $U$ is nonempty. This follows if we show that there is for all $i$ a $\phi \in H^0(\End_0(V_0)\otimes K_C)$ with $\op{res}_{p_i}(\phi) \neq 0$. The existence of such a twisted endomorphism can be deduced from Lemma \ref{lemma:bpf}: It gives us in particular $\phi$ such that $\phi_{21}^{(-1)}(p_i)\neq 0$. Now take $\phi \in U$. We have
\begin{equation*}
\delta^{-1}(q) = \{ A\phi A^{-1} : A \in GL_2(\Bbbk)\}.
\end{equation*}
We apply Corollary \ref{cor:nilpotent_fiber}: The morphism $A\phi A^{-1}$ is in $H^0(\End_0(V_0) \otimes K_C) $ for some $A\in GL_2(\Bbbk)$ if and only if $\op{res}_{p_i}(A\phi A^{-1})$ is nilpotent with kernel generated by $\begin{pmatrix}
	\xi(p_i) \\ 1
\end{pmatrix}$ for $i = 1,\ldots,d$. Because $GL_2(\Bbbk)$ acts three transitively on the projective line the conditions $A\begin{pmatrix}
\xi(p_i) \\ 1
\end{pmatrix} = \lambda_i \begin{pmatrix}
\xi(p_i) \\ 1
\end{pmatrix}
$ for $i=1,2,3$ fix the matrix $A$ to be a multiple of the identity. Here we are using that the $\xi(p_i)$ are general and in particular pairwise distinct. Therefore $\phi$ is the only preimage of $q$ in $U$ and $\delta_{|H^0(\End_0(V_0)\otimes K_C)}$ is generically injective. 

If $d = 3$ the proof of Proposition \ref{prop:trivial_bundle_mer} shows that $\op{ev}^{-1}_{|\op{im}\delta}(0)$ is only $0$. Hence $\op{ev}_{|Q}$ is quasi-finite and $h_V = \op{ev} \circ \delta$ is generically finite. 
Again by Proposition \ref{prop:trivial_bundle_mer} if $d = 4$ the intersection of the image of $\delta$ with $\ker(\op{ev})$ is one dimensional. This means there are finitely many quadrics $q_1,\ldots,q_N$ such that set-theoretically we have
\begin{equation*}
\op{im} \delta\cap \ker \op{ev} = \langle q_1\rangle \cup \ldots \cup \langle q_N \rangle.
\end{equation*}
If we prove that all $q_j$ are not in $Q$ then $\op{ev}_{|Q}$ is quasi-finite and the same argument as for $d = 3$ goes through. The idea is to attach an invariant in the moduli space of four pointed rational curves $\cM_{0,4}$ to every $q_j$. Choose $\phi_j\in H^0(\End_0(\OO^{\oplus 2}) \otimes K_C(\Gamma))$ such that $\delta(\phi_j) = q_j$. We think of $\ker(\op{res}_{p_i} \phi_j)$ as points of $\PP^1$. We observe that all matrices $\op{res}_{p_i}\phi_j$ are non-zero because otherwise we would find a rank $3$ quadric containing $C$ in $\PP H^0(K_C(\Gamma -q_k))$ for some $k$. Such a quadric does not exist as explained in the case $d = 3$. If for a given $j$ the points $\ker(\op{res}_{p_i}\phi_j), i=1,\ldots,4$ are pairwise distinct we can define the following invariant
\begin{equation*}
c_j := \left(\ker(\op{res}_{p_1} \phi_j),\ker(\op{res}_{p_2} \phi_j),\ker(\op{res}_{p_3} \phi_j),\ker(\op{res}_{p_4} \phi_j)\right) \in \cM_{0,4}.
\end{equation*}
If the kernels are not pairwise distinct, we say that $c_j$ is undefined. 
The element $c_j$ is independent of the choice of $\phi_j$: All preimages of $q_j$ under $\delta$ are of the form $A\phi_jA^{-1}$ with $A \in GL_2(\Bbbk)$. The kernels of $\op{res}_{p_i}A\phi_jA^{-1}$ change by the projective transformation $[A]$ and the point $c_j$ is the same.
Recall that $\cM_{0,4} \cong \Aff^1\setminus \{0,1\}$. In particular the moduli space is one dimensional and irreducible. Because $\xi$ is general the point $(\xi(p_1),\xi(p_2),\xi(p_3),\xi(p_4)) \in \cM_{0,4}$ is different from $c_j$. 

If some $q_j$ was in $Q$ we   find $\psi_j \in H^0(\End_0(V_0)\otimes K_C)$ such that $q_j = \delta(\psi_j)$. By Corollary \ref{cor:nilpotent_fiber} $\op{res}_{p_i}\phi_j$ is nilpotent with kernel generated by $ \begin{pmatrix}
 \xi(p_i) \\ 1 
\end{pmatrix}  $ for $i=1,\ldots,4$. Therefore the invariant $c_j$ is defined and we have 
\begin{equation*}
	c_j = (\xi(p_1),\xi(p_2),\xi(p_3),\xi(p_4)).
\end{equation*}
This is a contradiction to the previous paragraph.
\end{proof}
\section{Pencils on classical spectral curves}
\label{section:classical_spectral}
\begin{figure}[ht]
\tikzset{every picture/.style={line width=0.75pt}} %set default line width to 0.75pt        
\begin{tikzpicture}
%nodes
\node (start) [startstop] {$( V,\phi ) \in W^1_C( 2,d-2g+2)$};
\node (not stable) [startstop,below of=start,xshift=-2cm,yshift = -1.5cm] {$V\text{ is not stable}$};
\node (stable) [startstop,below of=start,xshift=2cm,yshift=-1.5cm] {$ \begin{array}{l}
V\text{ is stable}\\
d\ \geq 2g+1
\end{array}$};
\node (sss) [startstop,below of=not stable,xshift = 2.5cm,yshift=-1.5cm] {$ \begin{array}{l}
V\text{ is strictly semistable}\\
d\geq 2g-2
\end{array}$};
\node (unstable) [startstop,below of=not stable,xshift = -2.5cm,yshift=-1.5cm] {$V\text{ is unstable}$};
\node (unstablei) [startstop,below of=unstable,xshift = -2cm,yshift=-1.5cm] {$ \begin{array}{l}
h^{0}( L_{1}) \geq 2\\
d\geq g+1
\end{array}$};
\node (unstableii) [startstop,below of=unstable,xshift = 2cm,yshift=-1.5cm]{$ \begin{array}{l}
h^{0}( L_{1}) < 2\\
d\ \geq 2g-1
\end{array}$};
%arrows
\draw [arrow] (start) -- (stable);
\draw [arrow] (start) -- (not stable);
\draw [arrow] (not stable) -- (sss);
\draw [arrow] (not stable) -- (unstable);
\draw [arrow] (unstable) -- (unstablei);
\draw [arrow] (unstable) -- (unstableii);
\end{tikzpicture}
\centering
\caption{\label{fig:stratification} A line bundle $L \in W^1_d(\tilde{C})$ gives rise to a Higgs bundle $(V,\phi)$. The diagram shows the various possibilities for the vector bundle $V$.}
\end{figure}

In this section we investigate smooth spectral double covers $\psi: \tilde{C}\to C$ such that $\psi_*\OO_{\tilde{C}} \cong \OO_C \oplus K_C^\vee$. Throughout this section we assume the curve $C$ to be \textbf{Brill-Noether general} of genus $g \geq 3$. From Proposition \ref{prop:hitchin_corr} and equation (\ref{eq:deg_push}) we know that line bundles of degree $d+2g-2$ on $\tilde{C}$ correspond to Higgs bundles in $\cH_C(2,d)(K_C)$ whose characteristic polynomial is the defining polynomial of $\psi$. This is the original moduli space of Higgs bundles and therefore we will denote it simply by $\cH_C(2,d)$. 
We define Brill-Noether loci for the moduli space of Higgs bundles:
\begin{equation}
W^r_C(2,d) := \{(E,\phi) \in \cH_C(2,d) : h^0(E) \geq r +1\}.
\end{equation}
We now refine these stratifications by the type of the underlying vector bundle. For $V$ an unstable vector bundle of rank $2$ denote its composition factors in the Harder Narasimhan filtration by $L_1$ and $L_2$.  This means $L_1$ is the unique sub line bundle of maximal degree in $V$. In other words $V$ is an extension of $L_2$ by $L_1$.  We call $(\deg L_1,\deg L_2)$ the Harder-Narasimhan type of $V$. We define:
\begin{align}
W^r_C(2,d)_{s} &:= \{ (V,\phi) \in W^r_C(2,d) : \text{$V$ is stable}\} \\
W^r_C(2,d)_{sss} &:= \{ (V,\phi) \in W^r_C(2,d) : \text{$V$ is strictly semi stable}\} \\
W^{r_1,r_2}_C(2,d_1,d_2) &:= \left\{ (V,\phi) \in \cH_C(2,d) : 
\begin{aligned}
&\text{$V$ has Harder-Narasimhan type $(d_1,d_2)$,} \\
&h^0(L_i) \geq r_i +1\text{ for $i=1,2$}
\end{aligned}
\right\}.
\end{align}
In the definition of $W_C^{r_1,r_2}(2,d_1,d_2)$ we assume $d_1+ d_2 = d$  and $d_1 > d_2$. In this section we will content ourselves with the case $r=1$. Here the stratification is illustrated in Figure \ref{fig:stratification}. For a stratum to be non-empty the degree needs to above a certain number and this can be read off from diagram \ref{fig:stratification}. The degree bounds will be proved in this section. We start with the two loci $W_C^{1,-1}(2,d_1,d_2)$ and $W_C^{0,0}(2,d_1,d_2)$ in which the underlying vector bundles are unstable.
\begin{proposition}\label{prop:unstable_i}
Let $d_1,d_2$ be integers satisfying $0 <d_1 - d_2 \leq 2g-2$. If $\rho(g1,d_1) < 0$ the locus $W_C^{(1,-1)}(2,d_1,d_2)$ is empty and otherwise the following holds
\begin{equation}\label{eq:degree_diff_small}
\dim W_C^{1,-1}(2,d_1,d_2) = 
\begin{cases}
g-4 + d_1 + d_2 +4(g-1)+1&,d_1 \leq g+1\\
3g - d_1 + d_2 - 2 + 4(g-1)+1&, d_1 > g+1.
\end{cases}
\end{equation}
\end{proposition}
\begin{proof}
Every Higgs bundle $(V,\phi) \in W^{1,-1}_C(2,d_1,d_2)$ can be obtained from a Higgs bundle $(V',\phi) \in W^{0,-1}_C(2,0,d_2-1)$ by twisting with a line bundle in $W^1_{d_1}(C)$. This leads to the isomorphism:
\begin{equation*}
	W^{1,-1}_C(2,d_1,d_2) \cong W^{0,-1}_C(2,0,d_2-d_1)\times W^1_{d_1}(C).
\end{equation*}
Therefore we concentrate on the dimension of $W^{0,-1}_C(2,0,d)$ for $0 \leq d \leq 2g-2$.
 We will construct a surjective map from an irreducible scheme $M$ to $W^{0,-1}_C(2,0,-d)$ in several steps. Based on the size of $d$ we distinguish two cases in the proof.   

We start with $d_1 - d_2 < g-1$. We first explain which objects $M$ parametrizes: A point in $M$ is a tuple $(L,\xi,\phi)$, where $L$ is in $\op{Pic}^{-d}$$L_2 \in \op{Pic}^{d_2}(C)$. Furthermore $\xi$ is an extension of $L$ by $\OO_C$ inducing a vector bundle $V$. Furthermore $\phi$ is a twisted endomorphism such that $(V,\phi)$ is a stable Higgs bundle. In fact $M$ overparametrizes the extensions such that $W_C^{1,-1}(2,d_1,d_2)$ has positive dimensional fibers. The reason for this is explained in remark \ref{remark:ext}.
 
Now for the construction: Take $\cL$ to be a Poincaré line bundle on $\op{Pic}^{-d}(C)\times C$. We build an irreducible scheme $M_1$ together with a morphism $\pi_1: M_1 \to \op{Pic}^{-d}(C)$. On $M_1\times C$ we will have an extension 
\begin{equation}\label{eq:univ_extension}
0 \to \OO_{M_1 \times C} \to \cV_1 \to \cL \to 0.
\end{equation}
Here we denoted the pullback of $\cL$ along $M_1\times C \to \op{Pic}^{-d}(C) \times C$ by $\cL$ as well. The extension has the following property: If $L$ in in $\op{Pic}^ {d_1-d_2}$ and $\xi$ is an extension of $L$ by $\OO_C$ we find a closed point $m \in M_1 $ lying over $L$ such that (\ref{eq:univ_extension}) restricted to $\{m\}\times C$ gives the extension $\xi$. We build $M_1$ as a geometric vector bundle: Take $G$ a divisor of high degree on $C$ and put $\Gamma = \op{pr}^*_2G$ on $\op{Pic}^{d_1-d_2}(C)\times C$. By the vanishing of $R^1\op{pr}_{1*}\left(\cL^\vee(\Gamma)\right)$ we have a surjection 
\begin{equation*}
\mathcal{F}=\op{pr}_{1*}\left(\cL^\vee(\Gamma)/\cL^\vee\right) \to R^1\op{pr}_{1*}\cL^\vee.
\end{equation*}
If the degree of $G$ is high enough $\mathcal{F}$ is a vector bundle. We put $M_1 := \Aff(\mathcal{F})$ and denote the natural map from $M_1$ to $\op{Pic}^{-d}(C)$ by $\pi_1$. 
The dimension of $M_1$ equals $\dim \op{Pic}^{d_1-d_2} + \op{rk} \cF = g + \deg G$.

In the second construction step we use a relative section space which was constructed in subsection \ref{subsection:rel_section}:
\begin{equation*}
M_2 := S_{\op{pr}_1,\End(\cV_1)\otimes \op{pr}_2^*(K_C)}. 
\end{equation*} 
Denote its natural map to $M_1$ by $\pi_2$ and write $\cV_2 := (\pi_3\times \op{id})^*\cV_1$. It comes with a family of twisted endomorphisms $\Phi: \cV_2 \to \cV_2 \otimes \op{pr}_2^*K_C$. 
We compute the dimension of $M_2$. Let $r$ be the rank of $\op{pr}_{1*}\left( \End(\cV_2)\otimes \op{pr}_2^*(K_C)\right)^\vee$ at the generic point of $M_1$. From Lemma \ref{lemma.twisted_bound} we have  $r = 4(g-1)+1$ because the extension (\ref{eq:univ_extension}) is non-split at a general $m\in M_1$ and $h^0(\cL^\vee_{|m\times C}) = 0$ holds. We consider the following locally closed sets
\begin{equation}\label{eq:jump_loci}
M_1^{(k)} := \left\{ m \in M_1 :  h^0(C,\End\left(V_{1|m\times C}\right)\otimes K_C) = r+k\right\}. 
\end{equation} 
We show that $M_1^{(k)}$ has codimension at least $k$ to apply Lemma \ref{lemma:dimension_section_space}. This implies that the dimension of $M_2$ is $\dim M_1 + r$. Let $m$ be point of $M_1$.
Serre duality and the Riemann-Roch theorem show that $h^0(\End(\cV_{1|m\times C})) = 4(1-g) + h^0(\End(\cV_{1|m\times C}) \otimes K_C)$.
Hence we can rewrite 
\begin{align*}
M_1^{(k)} &= \left\{ m \in M_1: \dim \op{End}(\cV_{2|m\times C}) = 1+ k\right\}.
\end{align*}
The last identification follows from the Riemann-Roch theorem.
We use Lemma \ref{lemma.twisted_bound} 
to see that $M_1^{(k)}$ is the disjoint union of the following two subvarieties:
\begin{align*}
M_{1,a}^{(k)} &= \left\{m \in M_1: \begin{aligned}
&h^0(\cL_{|m\times C}^\vee) = k,\\ 
&\text{and $\cV_{1|m\times C}$ does not split}
\end{aligned} \right\} \\
M_{1,b}^{(k)} &= \left\{m \in M_1: 
\begin{aligned}
&h^0(\cL_{|m\times C}^\vee) =  k-1,\\ 
&\text{and $\cV_{1|m\times C}$ splits}
\end{aligned}
\right\}.
\end{align*} 
From the the dimension theorem of Brill-Noether theory we conclude that 
\begin{equation*}
\op{codim}M_{1,a}^{(k)} = \op{codim}_{\op{Pic}(C)}W^{k-1}_{d}(C) = k(g-1 -d +k) \geq k.
\end{equation*}
For $M_{1,b}^{(k)}$ we obtain the following  
\begin{align*}\label{eq:dim_jump_loci}
\op{codim}_{M_1} M_{1,b}^{(1)} &= g -1 -d\geq 1,\\
\op{codim}_{M_1} M^{(k)}_{1,b} &= \op{codim}_{\op{Pic}(C)}W^{k-2}_{d_1-d_2}(C) + g-1 -d + k -1 \\
&= k(g-2 -d + k)\geq k \quad\text{for $k\geq 2$}.
\end{align*}
To see this note first the inclusion $M^{(k)}_{1,b} \subset \pi_1^{-1}\left(W^{k-2}_{d_1-d_2}(C)\setminus W^{k-1}_{d_1-d_2}(C)\right)$ for $k\geq 2$.
The fiber of  $L \in W^{k-2}_{d_1-d_2}(C)\setminus W^{k-1}_{d_1-d_2}(C)$ in $M^{(k)}_{1,b}$ consists of the kernel of $H^0(L(G)) \to H^1(L)$ because all extensions in $M^{(k)}_{1,b}$ split. Therefore $M^{(k)}_{1,b}$ has codimension $h^1(L) = g-1 -d + k -1 $ in $\pi_1^{-1}\left(W^{k-2}_{d_1-d_2}(C)\setminus W^{k-1}_{d_1-d_2}(C)\right)$. We have shown that $M_2$ has dimension 
\begin{equation}
\dim M_2 = \dim M_1 + 4(g-1) +1 = g + \deg G+ 4(g-1) +1.
\end{equation}

On $M_2\times C$ we have a twisted endomorphism $\Phi: \cV_2 \to \cV_2 \otimes \op{pr}_2^*K_C$. It induces a surjective rational map $M_2 \dashrightarrow W_C^{0,-1}(2,0,-d)$. We define $M$ to be the open set where this rational map is defined. In other words $M$ consists of the points of $M_2$ corresponding to stable Higgs bundles. 
 We next find the dimension of the generic fiber of $M \dashrightarrow W_C^{0,-1}(2,0,-d)$ . A point of $M$ is a $3$-tuple 
\begin{equation}
(L\in \op{Pic}^{-d}(C),\xi \in H^0(L^\vee(G)_{|G}),\phi \in H^0(\End(V_\xi)\otimes K_C)).
\end{equation}
Two elements $\xi_1,\xi_2 \in H^0(L^\vee(G)_{|G})$ give rise to the same vector bundle if their image under the boundary map $H^0(L^\vee(G)_{|G}) \to H^1(L^\vee)$ differ by multiplication with a nonzero scalar. Furthermore a second twisted endomorphism $\phi'$ of $V_\xi$ induces an equivalent Higgs bundle if and only if there is an automorphism $\eta$ of $V_\xi$ such that $\phi' = \eta\phi \eta^{-1}$. But for general point in $M$ the bundle $V_\xi$ has no nontrivial automorphisms as we have seen in the computation of $r$. Hence the dimension of the general fiber equals 
\begin{align}\label{eq:redundancy}
h^0(L^\vee(G)_{|G})-h^1(L^\vee) + 1.
\end{align}
Hence we find the dimension of $W^{0,-1}_C(2,0,-d)$ to be
\begin{align*}
&\dim M - h^0(L^\vee(G)_{|G})+h^1(L^\vee) - 1 = g-2 - d + 4(g-1) + 1.
\end{align*}

In the second part of the proof assume $d \geq g-1$. We build a scheme with a surjective map to $W_C^{0,-1}(2,0,d_2-d_1)$. As remarked in the proof of Proposition \ref{lemma:max_subbundle} a vector bundle obtained as an extension of $L$ by $\OO_C$, where $\deg L_1 > \deg L_2$ only admits a twisted endomorphism producing a stable Higgs bundle if $h^0(L\otimes K_C) > 0$. This is equivalent to the existence of a nontrivial extension of $L_2$ by $L_1$. We abbreviate $W^{1-g+d}_{d}(C)$ by $M_0$. It has dimension $2g-2-d$. A line bundle $L$ of degree $d$ is in there if and only if $h^0(K_C\otimes L^\vee) > 0$. We fix a Poincaré line bundle $\cL$ on $M_0 \times C$. 
We write $\Gamma = \op{pr}_{2}^*G$ and $\cF := \op{pr}_{1*}\left(\cL(\Gamma)/\cL\right)$. Then we set $M_1 = \Aff(\cF)$. On $M_1 \times C$ we again have a vector bundle $\cV_1$ of rank $2$ which is an extension of $\cL$ by $\OO_C$. We note
\begin{equation*}
\dim M_1 = 2g-2 -d + \deg G.
\end{equation*}
In the next step let $M_2$ be the relative section space $ S_{\op{pr}_1,\op{End}(\cV)\otimes \op{pr}_2^*K_C}$. We compute its dimension via jump loci which are defined as in (\ref{eq:jump_loci}). Here the generic rank of $h^0(\End(\cV_{m\times C}\otimes K_C)$ is $3(g-1)+d+2$. We show that $\op{codim} M_1^{(k)} \geq k$.
A compution with Riemann-Roch and Lemma \ref{lemma.twisted_bound} shows that $M^{(k)}_1$ is the disjoint union of the following two loci:
\begin{align*}
M_{1,a}^{(k)} &= \left\{m \in M_1: \begin{aligned}
&h^0(\cL_{|m\times C}\otimes K_C) = k+1,\\ 
&\text{and $\cV_{1|m\times C}$ does not split}
\end{aligned} \right\} \\
M_{1,b}^{(k)} &= \left\{m \in M_1: 
\begin{aligned}
&h^0(\cL_{|m\times C}\otimes K_C) = k,\\ 
&\text{and $\cV_{1|m\times C}$ splits}
\end{aligned}
\right\}.
\end{align*}
The subvariety $M_{1,a}^{(k)}$ is open in the preimage of $W^k_{2g-2-d}(C) \subset M_0$. Hence we find
\begin{align*}
\op{codim} M^{(k)}_{1,a} &= \op{codim}_{W^0_{2g-2-d}(C)}W^k_{2g-2-d}(C) \\
&= \rho(g,0,2g-2-d) - \rho(g,k,2g-2-d) \\
&= (k+1)(g-2g+2+d+k) - (g-2g+2+d)\\
&= (k+1)k + k(-g+2+d)\\
&\geq (k+1)k + k(-g+1+g-1)
\geq (k+2)k \geq k+1.
\end{align*}
Therefore $M^{(k)}_{1,a}$ has enough codimension. For $M^{(k)}_{2,b}$ we have
\begin{align*}
\op{codim} M^{(k)}_{2,b} = \op{codim}_{W^0_{2g-2-d}(C)}W^{k-1}_{2g-2-d}(C) + k \geq k.
\end{align*}
Now the generic rank of $\End(\cV)\otimes \op{pr}_2^*(K_C)$ is by equation (\ref{eq:dim_endos}) equal to $4(g-1) + 1 + h^0(\cL_{|m\times C}) = 3(g-1)+ 1 +2+d$. Here $L$ is a general element of $M_0$. We conclude
\begin{equation}
\dim M_2 = g +\deg G + 4(g-1)+1.
\end{equation}
As in the first part of the proof there is a surjective rational map $M_3 \dashrightarrow W_C^{0,-1}(2,0,-d)$. We use equation (\ref{eq:redundancy}) to find the dimension of $W_C^{0,-1}(2,0,-d)$:
\begin{align*}
&W_C^{0,-1}(2,0,-d) \\
&= \dim M_3 -h^0(L(G)/L)+h^1(L) -1 - h^0(\op{End}_0(V_\xi)) \\
&= g-2 - d + 4(g-1) + 1.
\end{align*}
To finish the proof we use the following sum: 
\begin{equation*}
	\dim W_C^{1,-1}(2,d_1,d_2) = \dim W_C^{0,-1}(2,0,d_2-d_1) + \dim W^1_{d_1}(C).
\end{equation*}
Note that the dimension of $W^1_{d_1}(C)$ is $-g-2+2d_1$ if $d_1 \leq g$ and it is $g$ if $d_1 \geq g+1$.
\end{proof}
\begin{remark}\label{remark:ext}
Some complications in the above proof arise because the following situation arises: We are given a family of vector bundles on a curve $C$ i.e a vector bundle $\cF$ on $C \times T$ where $T$ is a scheme. We would like to build a moduli space out of the different extension spaces $\op{Ext}^1(\cF_{|t \times C},\OO_{C})$ where $t$ varies in $T$. The resulting moduli space is then naturally an Artin stack because the automorphism group of a given extension $\xi \in \op{Ext}^1(\cF_{|t \times C},\OO_{C})$ is isomorphic to $\op{Hom}(\cF_{m\times C},\OO_C)$. To keep the approach more elementary we do not pursue this point of view.
\end{remark}
The following lemma will be needed only for the next Proposition:
\begin{lemma}\label{lemma:difference_loci}
Let $0\leq d_2\leq d_1\leq g-1$ between $0$ and $g-1$. Let $\alpha$ be the difference map $C_{d_1}\times C_{d_2}\to \op{Pic}^{d_1-d_2}(C), (D_1,D_2)\mapsto \OO_C(D_1-D_2)$. For $r\geq 1$ the following holds: No component of $\alpha^{-1}(W^{r-1}_d(C))$ is entirely contained in $\alpha^{-1}(W^r_d(C))$. In particular we have
 \begin{equation*}
\op{codim} \alpha^{-1}(W^r_d(C)) \geq r+1.
\end{equation*} 
\end{lemma}
\begin{proof}
Take $(D_1,D_2)\in C_{d_1}\times C_{d_2}$ such that $h^0(\OO_C(D_1 - D_2)) = r+1$. We are going to perturb $(D_1,D_2)$ slightly to $(D_1',D_2')$ such that $h^0(\OO_C(D_1'-D_2')) = r$. Note that either $h^0(\OO_C(D_1)) < r+1 + d_2$ or $h^0(\OO_C(D_1)) = r+1 + d_2$ hold. In the first case some point $q$ of $D_2$ is in the base locus of the linear system $|D_1|$. Move $q$ outside of the base locus to obtain $D_2'$. Then $h^0(\OO_C(D_1 - D_2')) = r$. 

In case $h^0(\OO_C(D_1)) = r+1+d_2$ we perturb $D_1$ to $D_1'$ such that $h^0(\OO_C(D_1')) = r+ d_2$. Here we use that in general no irreducible component of $W^{r+d_2-1}_{d_1}(C)$ is contained in $W_{d_1}^{r+d_2}(C)$. See \cite[chapter IV, Lemma 3.5]{ACGH}. We can assume that no point of $D_2$ is a base point of $|D_1'|$. Hence $h^0(\OO_C(D_1'-D_2)) = r$.   

The lower bound on the codimension follows by induction on $r$: Assume $\alpha^{-1}(W^{r-1}_d(C))$ has codimension at least $r$. Because no component of $\alpha^{-1}(W^{r-1}_d(C)) $ is contained entirely in $\alpha^{-1}(W^r_d(C))$ we conclude that 
\begin{equation*}
\op{codim} \alpha^{-1}(W^r_d(C)) \geq 1+ \op{codim} \alpha^{-1}(W^{r-1}_d(C))\geq r+1.
\end{equation*}
\end{proof}
\begin{proposition}\label{prop:unstable_ii}
For $d_1 + d_2 < 1$ the locus $W_C^{0,0}(2,d_1,d_2)$ is empty. If $0 \leq d_2 < d_1 \leq g$ we have
\begin{equation*}
\dim W_C^{0,0}(2,d_1,d_2) = d_1 + 2d_2 -2 + 4(g-1)+1.
\end{equation*}
If $g+1\leq d_1$ we have a containment: $W_C^{0,0}(2,d_1,d_2) \subset W_C^{1,-1}(2,d_1,d_2)$.
\end{proposition}
\begin{proof}
Note first that effective line bundles have degree $\geq 0$ hence $d_2 \geq 0$ if $W^{(0,0)}_C(2,d_1,d_2)$ is non-empty. Furthermore we have $d_1 > d_2$. Hence $W^{(0,0)}_C(2,d_1,d_2)$ is empty if $d_1 + d_2 < 1$.

Now assume $0 \leq d_2< d_1 \leq g$.
We build a family of Higgs bundles over a scheme $M$ such that the induced map $M \to W_C^{0,0}(2,d_1,d_2)$ is surjective. We start with $M_1 = C_{d_1}\times C_{d_2}$. On $M_1 \times C$ we have two divisors $\cD_1,\cD_2$ which satisfy:
\begin{align*}
 \cD_{i|(D_1,D_2)\times C} = D_i
\end{align*}
for all $(D_1,D_2) \in M_1$ and $i =1,2$. We define $M_2$ to be $\Aff(\op{pr}_{1*}\OO(\cD_1)_{|\cD_2})$. Its dimension is $d_1 + 2d_2$. On $M_2 \times C$ we have an extension 
\begin{equation}\label{eq:univ_extension_ii}
0 \to \OO_{M_2\times C}(\cD_1) \to \cV \to \OO_{M_2\times C}(\cD_2) \to 0.
\end{equation}
 From Lemma \ref{lemma:lift_section} it follows that $h^0(\cV_{2|m\times C}) = 2$ for all $m \in M_2$. Put $M_3 = S_{\op{pr}_1,\End(\cV)\otimes \op{pr}_2^*(K_C)}$. We compute the dimension of $M_3$. We employ Lemma \ref{lemma:dimension_section_space}: To do this we need to find an upper bound for the dimension of the following loci:
\begin{equation*}
M_2^{(k)} := \left\{ m \in M_2 :  h^0(C,\op{End}\left(\cV_{1|m\times C}\right)\otimes K_C) = 4(g-1)+1+k\right\},\quad k\in \NN. 
\end{equation*}
By Riemann-Roch $m\in M_2$ is in $M_2^{(k)}$ if and only if $h^0(\End(\cV_{|m}))= k+1$. We decompose this locus in two parts:
\begin{align*}
M_{2,a}^{(k)} &= \left\{(D_1,D_2,\tilde{\xi}) \in M_2 : \begin{aligned}
&h^0(\OO(D_1-D_2)) = k \\
&\text{and the extension does not split}
\end{aligned} \right\},\\
M_{2,b}^{(k)} &= \left\{(D_1,D_2,\tilde{\xi}) \in M_2 : \begin{aligned}
&h^0(\OO(D_1-D_2)) = k-1 \\
&\text{and the extension splits}
\end{aligned} \right\}.
\end{align*}
One sees from Lemma \ref{lemma:difference_loci} that $M^{(k)}_{2,a}$ has codimension greater or equal to $k$. For $M^{(k)}_{2,b}$ we need to take the splitting into account. An element $\tilde{\xi}\in H^0(\OO(D_1)_{|D_2})$ induces the trivial extension if and only if it is in the kernel of $H^0(\OO(D_1)_{|D_2})\to H^1(\OO(D_1-D_2))$. Using a long exact sequence we see that the image has in general codimension $\deg D_2 -1$. 
We conclude 
\begin{align*}
\op{codim} M_{2,b}^{(k)} \geq k+d_2 - 2.
\end{align*}
If $d_2 \geq 2$ we have shown that $M_3$ has dimension $d_1 +2d_2 + 4(g-1)+1$. We will handle the cases $d_2 = 0$ and $d_2 = 1$ separately at the end of the proof. 

Let $M$ be the open subscheme of $M_3$ parametrizing stable Higgs bundles. We find the dimension of the image of $M\to \cH_C(2,d)$. Recall that an element of $M$ is a tuple
\begin{align*}
(D_1\in C_{d_1},D_2\in C_{d_2},\tilde{\xi} \in H^0(\OO(D_1)_{|D_2}),\phi).
\end{align*}
Here $\tilde{\xi}$ induces an extension with vector bundle $V$ and $\phi$ is a twisted endomorphism of $V$. Note that the morphism $H^0(\OO(D_1)_{|D_2}) \to H^1(\OO(D_1-D_2))$ has in the general case a one dimensional kernel $H^0(\OO(D_1))$. Furthermore we can multiply $\tilde{\xi}$ by a nonzero scalar and we obtain an isomorphic Higgs bundle. Therefore we have 
\begin{align}\label{eq:dimension_0,0unstable}
\dim W_C^{0,0}(2,d_1,d_2) &= \dim M - 2 \\
&= d_1 + 2d_2 -2 + 4(g-1)+1.
\end{align}

If $d_2 = 0$ all extensions parametrized by $M_2$ split because $D_2 = 0$ and therefore $H^0(\OO(D_1)_{|D_2})$ is zero too. Therefore $M^{(k)}_{2}$ equals $M^{(k)}_{2,b}$. This locus has codimension $\geq k-2$, where $k \geq 2$. We conclude that $M_3$ has dimension ${2d_2 +2  +4(g-1)+1}$. From lemma \ref{lemma:aut_vectorbundle} we see that for a general $m\in M_2$ we have $h^0(\End(\cV_{|m})) = 3$. This implies that the morphism $M \to \cH_C(2,d)$ generically has fibers of dimension $4$. Therefore $W_C^{0,0}(2,d_1,0)$ has dimension $d_1 -2  + 4(g-1)+1$. 

In case $d_2 = 1$ the extension generically splits. This can be seen as follows: Let ${(D_1,D_2,\tilde{\xi} \in H^0(\OO(D_1)_{|D_2}))}$ be a general point of $M_2$. This implies that $\OO(D_1-D_2)$ is not effective. Then we look at the long exact sequence of $0 \to \OO(D_1-D_2) \to \OO(D_1) \to \OO(D_1)_{|D_2} \to 0$:
\begin{equation*}
0 = H^0(\OO(D_1-D_2)) \to H^0(\OO(D_1)) \to H^0(\OO(D_1)_{|D_2}) \to H^1(\OO(D_1-D_2)).
\end{equation*} 
The map in the middle is an isomorphism and therefore the right map is $0$. Hence $\tilde{\xi}$ induces a split extension. One deduces that $M_3$ has dimension ${d_1 + 3 + 4(g-1) +1 }$. The group of automorphisms of $\OO(D_1)\oplus \OO(D_2)$ is in general two dimensional. Hence the fibers of $M\to \cH_C(2,d)$ are generically of dimension $3$. Therefore $W_C^{0,0}(2,d_1,d_2)$ has dimension $d_1 + 4(g-1) +1$ and $(V,\phi)$ is in $W_C^{1,-1}(2,d_1,d_2)$. 

For the last claim in the proof take a stable Higgs bundle $(V,\phi)$ with Harder-Narasimhan filtration $L_1 \subset V$ such that $\deg L_1 \geq g+1$. By the Riemann-Roch inequality we have $h^0(L_1)\geq 2$.
\end{proof}
Next we treat the relevant Higgs bundles where the underlying vector bundle is strictly semistable. 
\begin{proposition}
Assume $d$ is even. If $d$ is negative, then $W_C^1(2,d)_{sss} $ is empty. Furthermore we have 
\begin{align*}
\dim W_C^1(2,d)_{sss} = \begin{cases}
4(g-1)+1 &,d = 0\\
\frac{3}{2}d-2 + 4(g-1)+1&,2 \leq d \leq g+1\\
g + d -4 +4(g-1)+1&, g+2 \leq d \leq 2g-2.
\end{cases}
\end{align*}
\end{proposition}
\begin{proof}
Throughout the proof $V$ will denote a strictly semistable vector bundle of degree $d$. Therefore it has a line subbundle of degree $\frac{d}{2}$. In other words it can be obtained as an extension of two line bundles $L_1$ and $L_2$ of degree $\frac{d}{2}$. The inclusion of $L_1$ into $V$ is a Jordan-Hölder filtration and therefore the composition factor $L_1$ and $L_2$ are uniquely determined. See \cite[Proposition 5.3.7]{LePotier}. At least one of them is effective. This implies $d \geq 0$. 

We start with $d = 0$. A vector bundle with two linearly independent global sections needs to be an extension of $\OO_C$ by itself. Then Lemma \ref{lemma:lift_section} implies that only the trivial extension $V = \OO_C^{\oplus 2}$ has two linearly independent global sections. There is a $4(g-1)+1$ dimensional family of Higgs bundles of the form $(\OO_C^{\oplus 2},\phi)$.

Next we handle the case $d = 2$. The vector bundle $V$ is an extension of $\OO_C(q)$ by $\OO_C(p)$ for suitable points $p,q \in C$. If $p \neq q$ we need to have $V \cong \OO_C(p) \oplus \OO_C(q)$ by Lemma \ref{lemma:lift_section}: Indeed $q$ is a base point of the map $C \dashrightarrow \PP H^0(K_C(q-p))$. For $p \neq q$ we have a $2 + 4(g-1) +1$ dimensional family of Higgs bundles. If $p = q$ there is by Lemma \ref{lemma:lift_section} up to a scalar a unique extension $V$ with $h^0(V) = 2$. We get two families of Higgs bundles which both have dimension $4(g-1)+2$: In one case the underlying vector bundle is always split and in the other one it is non split. For the generic member in both families the underlying vector bundle has a two dimensional automorphism group. Hence the locus $W_C^1(d,2)$ has dimension $1+4(g-1)+1$. 

Now we take $d$ in the range of $4\leq d \leq g+1$. Note that $\rho(g,1,\frac{d}{2}) < 0$ hence the factors in the Harder-Narasimhan filtration $L_1$ and $L_2$ both need to be effective. We are now going to build a scheme $M$ which comes with a family of Higgs bundles in $W_C^1(2,d)_{sss}$. The construction is very similar to one in the proof of Proposition \ref{prop:unstable_ii}. Set $M_1 = C_{d/2}\times C_{d/2}$. We have two universal divisors $\cD_1$ and $\cD_2$ on $M_1\times C$. They have the following property:
\begin{equation*}
\cD_{i|(D_1,D_2)\times C} = D_i
\end{equation*} 
for all $(D_1,D_2) \in M_1$ and $i=1,2$. Then we put $M_2 = \Aff(\op{pr}_{1*}\OO(\cD_1)_{|\cD_2})$. For notational convenience we denote the pull backs of $\cD_1$ and $\cD_2$ to $M_2 \times C$ by the same letters. We have an extension on $M_2 \times C$:
\begin{equation}
0 \to \OO(\cD_1) \to \cV \to \OO_C(\cD_2) \to 0.
\end{equation}
The vector bundle $\cV$ has the following property: For all points $m$ of $M_2$ we have $h^0(C,\cV_{|m\times C}) \geq 2$. The vector bundle $\cV_{m\times C}$ might not be semistable for all $m\in M_2$. We replace $M_2$ by its open subscheme corresponding to semistable vector bundles. We then define $M$ to be 
\begin{equation*}
M_3 = S_{\op{End}(\cV)\otimes \op{pr}_2^*(K_C),\op{pr}_1}.
\end{equation*}
We compute the dimension of $M_3$. First $M_1$ has dimension $d$. Then $M_2$ has dimension $\frac{3}{2}d$ because $\op{pr}_{1*}\cL_1/\cL_1(-\cD_1)$ has rank $\frac{d}{2}$. Finally $M_3$ has dimension $\frac{3}{2}d+4(g-1) +1$: A general $m \in M_2$ satisfies $h^0(\End_0(\cV_{|m\times C})) = 1$. This implies $h^0(\End_0(\cV_{|m\times C}\otimes K_C)) = 4(g-1)+1$. The locus in $M_2$ where the vector bundle $\cV_{|m\times C}$ has $2$ endomorphisms has codimension $\geq 1$ as the reader can quickly check. On $M_3\times C$ we have a family of Higgs bundles $\Phi: \cV\to \cV\otimes \op{pr}^*_2(K_C)$. This induces a rational map from $M_3$ to the moduli space of Higgs bundles. We let $M$ be the open subscheme of $M_3$ where this map is defined. By construction the map $M \to W_C^1(2,d)_{sss}$ is dominant. We find the dimension of the image: One checks that the the arguments preceding equation (\ref{eq:dimension_0,0unstable}) are still valid in these case to find 
\begin{align*}
\dim W_C^1(2,d)_{sss} &= \dim M_3 - 2\\
&= \frac{3}{2}d-2+4(g-1)-1
\end{align*}

In the second part of the proof we take $d$ in the range $g+2\leq d \leq 2g-2$. For these degrees we can still define the family of Higgs bundles $M$ and the formula for their dimension stays valid but the map $M \to W_C^1(2,d)_{sss}$ is no longer surjective. To see this we take $V$ to be a strictly semistable vector bundle of degree $d$ on $C$ with $h^0(V)\geq 2$. 
There is an exact sequence $0 \to L_1 \to V \to L_2 \to 0$ where $L_1,L_2$ are line bundles both of degree $\frac{d}{2}$. There are three possibilities:
\begin{enumerate}
\item Both $L_1$ and $L_2$ are effective. 
\item The dimension of $H^0(L_1)$ is $2$.
\item The dimension of $H^0(L_2)$ is $2$.
 \end{enumerate}
All Higgs bundles in the first case are covered by the family $M$. 
We will build a second family of Higgs bundles $N$ which covers the second case. 
To construct $N$ we start with $N_1 = W^1_{d/2}(C)\times \op{Pic}^{d/2}(C)$. Note that $W^1_{d/2}(C)$ is nonempty because $\rho(g,1,d/2) \geq 0$. On $N_0\times C$ we have Poincaré line bundles $\cL_1$ and $\cL_2$. They satisfy the following 
\begin{equation*}
\cL_{1|(L_1,L_2)\times C} \cong L_1, \cL_{2|(L_1,L_2)\times C} = L_2
\end{equation*} 
for all $(L_1,L_2) \in N_0$. We write $\cL = \cL_1 \otimes \cL_2^\vee$. Furthermore we choose a divisor $G$ of high degree on $C$ and we set $\Gamma = N_1\times G$. Then the boundary map $\op{pr}_{1*}\cL/\cL(\Gamma) \to R^1\op{pr}_{1*}\cL$ is surjective. We define $N_2 = \Aff(\op{pr}_{1*}\cL/\cL(\Gamma))$. On $N_2 \times C$ we have an extension 
\begin{equation*}
0 \to \cL_1 \to \cV \to \cL_2 \to 0.
\end{equation*}
The space $N_2$ parametrizes all extensions of $L_2$ by $L_1$, where $(L_1,L_2) \in N_1$. Again we make $N_2$ a little bit smaller to ensure that $\cV_{n\times C}$ is semistable for all $n\in N$. We define $N_3 $ as a relative section space:
\begin{equation}
N_3 = S_{\op{End}(\cV)\otimes \op{pr}_2^*(K_C),\op{pr}_1}.
\end{equation} 
Similar to $M_3$ we have a family of Higgs bundles on $N_3$. We define $N$ to be the open subscheme on which the map $M_3 \dashrightarrow \cH_C(2,d)$ is defined. 
We now compute the dimension of the image of $N \to \mathcal{H}_C(2,d)$. First $N_1$ has dimension $\rho(g,d/2,1) + g$. Then $N_2 $ has dimension $\rho(g,d/2,1) + g +\deg G$. We can use Lemma \ref{lemma:dimension_section_space} to show that
\begin{equation*}
\dim N = \rho(g,d/2,1) + g +\deg G + 4(g-1) +1.
\end{equation*} 
Now an element of $M$ is a tuple $(L_1,L_2,\tilde{\xi} \in H^0(\left(L_1\otimes L_2^\vee\right)_{|G}),\phi)$. 
Another element $\tilde{\xi}'$ gives the same vector bundle as $\tilde{\xi}$ if their images in $H^1(L_1\otimes L_2)^\vee$ agree up to scalar multiplication. For general $L_1$ and $L_2$ we have $h^1(L_1\otimes L_2^\vee) = g-1$. Hence the general fiber of $N \to \cH_C(2,d)$ has dimension $ - g +2+\deg G$. Therefore the dimension of the image of $N$ equals 
\begin{align*}
&\rho(g,d/2,1)+g +(g-2) +4(g-1)+1 \\
&= g + d -4 +4(g-1)+1.
\end{align*} 
We need a third family $O$ of Higgs bundles for case 3 above: $O$ is constructed in the same way as $N$ except that in the first step we take $O_1 = \op{Pic}^{d/2} \times W^1_{d/2}(C)$. The image of $O$ contains Higgs bundles $(V,\phi)\in \cH_2(2,d)$ where $V$ is strictly semistable and can be obtained as an extension from $L_2 \in W^1_{d_1}(\tilde{C})$ by $L_1 \in W^1_{d_2}(\tilde{C})$. If the family $O$ is non-empty its image in $\cH_C(2,d)$ has exactly the same dimension as that of $N$. But note that not every vector bundle $V$ of the form above need to satisfy $h^0(V)\geq 2$. Not every section of $L_2$ might lift to a section of $V$. In other words the image of $O$ is not necessarily contained in $W_C^1(2,d)_{sss}$. In the end we obtain that the dimension of $W_C^1(2,d)_{sss}$ equals
\begin{equation*}
\op{max}\left\{\frac{3}{2}d-2+4(g-1)+1,g + d -4+4(g-1)+1\right\} = g+d-4+4(g-1)+1.
\end{equation*}  
for $g+1\leq d \leq 2g-1$. 
\end{proof}
Now we treat the locus of stable bundles $W_C^1(2,d)_s$.
\begin{proposition}
For $< 3$ the locus $W^1_C(2,d)_s$ is empty and for $3\leq d \leq 2g-1$ we have 
\begin{equation}
\dim W_C^1(2,d)_{s}  = 2d -3 + 4(g-1) +1.
\end{equation}
\end{proposition}
\begin{proof}
The following dimension is computed in \cite{Teixidor}: 
\begin{equation*}
\dim W^1_C(2,d) = 2d -3. 
\end{equation*}
This holds if $3 \leq d \leq 2g-1$. Consider the rational forgetful map $\pi: \cH_C(2,d) \dashrightarrow \mathcal{U}_C(2,d)$. The locus $W_C^1(2,d)_{s}$ is the set-theoretic inverse image of $W^1_C(2,d)$ under $\pi$. For a stable vector bundle we have $h^0(C,\End(E) \otimes K_C) = \chi(C,\End(E) \otimes K_C) + h^0(C,\End(E) \otimes K_C) = 4(g-1) +1$ and this is the fiber dimension of $\pi$ for each stable vector bundle. Adding this to $2d-3$  we obtain the result. 
\end{proof}

Lets collect our results in a theorem about Brill-Noether loci of spectral curves:
\begin{theorem}\label{thm:bn_strata}(Theorem B) Let $C$ be a general curve of genus $g\geq 3$. Let $\psi: \tilde{C} \to C$ be a spectral cover corresponding to a general element $(s_1,s_2)$ of the Hitchin base $H^0(K_C)\oplus H^0(K_C^2)$. Then the following statements hold:
\begin{itemize}
\item[(i)] The locus $W^{1,-1}_{d_1,d_2}(\tilde{C})$ is not empty if and only if $0 < d_1 - d_2 \leq 2g-2$ and $\rho(g,1,d_1,1) \geq 0$. In this case we have 
\begin{equation*}
\dim W^{1,-1}_{d_1,d_2}(\tilde{C}) = 
\begin{cases}
g-4 + d_1 + d_2 &,d_1 \leq g+1\\
3g - 2- d_1 + d_2  &, d_1 > g+1.
\end{cases}
\end{equation*} 
\item[(ii)] The locus $W^{0,0}_{d_1,d_2}(\tilde{C})$ is empty if $d_2 < 0$ or $d_1 - d_2 > 2g-2$ hold. If $0 \leq d_2 \leq g, 0\leq d_1 - d_2 \leq 2g-2$ we have the following upper bound for its dimension:
\begin{equation}
\dim W^{0,0}_{d_1,d_2}(\tilde{C}) \leq d_1 + 2d_2 -2.
\end{equation}
If $d_1 \geq g+1$ we have the containment $W^{0,0}_{d_1,d_2}(\tilde{C}) \subset W^{1,-1}_{d_1,d_2}(\tilde{C})$
\item[(iii)] The locus $W^1_d(\tilde{C})_{sss}$ is not empty if and only $d\geq 2g-2$ and $d$ is even. Furthermore we have for the dimensions
\begin{equation*}
\dim W^1_d(\tilde{C})_{sss} = \begin{cases}
0 &,d=2g-2\\
\frac{3}{2}d - 3g +1&,2g\leq d \leq 3g-1,\\
d - g -2&,3g \leq d \leq 4g-5.
\end{cases}
\end{equation*}
\item[(iv)] The locus $W^1_d(\tilde{C})_{s}$ is not empty if and only if $d\geq 2g+1$. In this case we have 
\begin{equation*}
\dim W^1_d(\tilde{C})_{s} = 2d -4g +1,\quad 2g+1 \leq d \leq 4g-3.
\end{equation*}
\end{itemize}
\end{theorem}
\begin{proof}
To prove (i) we first show that the Hitchin map restricted to $W^{1,-1}_C(2,d_1,d_2)$
\begin{equation*}
W_C^{1,-1}(2,d_1,d_2) \to H^0(K_C) \oplus H^0(K_C^2)
\end{equation*} 
is dominant. Let $(s_1,s_2)$ be any element of $H^0(K_C)\oplus H^0(K_C^2)$. Take $L_1$ to be a $g^1_{d_1}$. Factorize $s_2$ into $s_2's_2''$, where $s_2',s_2''$ are global sections of line bundles of degree $d_2 - d_1+2g-2$ and $d_1 - d_2 + 2g -2$. respectively. The vector bundle $E = L_1 \oplus (L_1 \otimes K_C^\vee(\div s_2'))$ has the correct degrees in the Harder-Narasimhan filtration. On $E$ define the twisted endomorphism:
\begin{equation*}
\phi = \begin{pmatrix}
-s_1 & -s_2'' \\
s_2' & 0
\end{pmatrix} : L_1 \oplus (L_1 \otimes K_C^\vee(\div s_2')) \to (L_1 \otimes K_C) \oplus L_1(\div s_2').
\end{equation*}
It has characteristic polynomial $T^2 + s_1T + s_2$. This shows the dominance. Then a generic fiber of the Hitchin map on $W_C^{1,-1}(2,d_1,d_2)$ has dimension $\dim W_C^{1,-1}(2,d_1,d_2) - 4(g-1)+1$. Plugging in the numbers from Proposition \ref{prop:unstable_i} we conclude (i). 

We come to (ii). If the Hitchin map restricted to $W_C^{0,0}(2,d_1,d_2)$ is surjective we clearly have the upper bound stated above by Proposition \ref{prop:unstable_ii}. Otherwise $W^{0,0}_{d_1,d_2}(\tilde{C})$ is empty for a general curve $\tilde{C}$.
To prove (iii) we show that the map $W_C^{1,-1}(2,d)_{sss} \to H^0(K_C)\oplus H^0(K_C^2)$ is dominant if $d$ is non-negative and even. Pick any effective line bundle of degree $\frac{d}{2}$ and set $V = L^{\oplus 2}$. By Proposition \ref{prop:trivial_bundle_mer} and the discussion preceding it the map $H^0(\End(V)\otimes K_C) \to H^0(K_C)\oplus H^0(K_C^2)$ is surjective. This shows dominance.

For the proof of (iv) we show that the Hitchin map restricted to $W_C^1(2,d-2g+2)_{s}$ is dominant for $d-2g+2 \geq 3$. Equivalently we show that a general spectral curve $\tilde{C}$ admits a line bundle of degree $d$ with $h^0(L)\geq 2$ and $\psi_*L$ stable. If $d-2g+2 = 3,4$ there is a vector bundle in $W^1_C(2,d-2g+2)$ such that $h_V$ is dominant as shown in Lemma \ref{lemma:restricted_hitchin_3,4}.
Assume $d-2g+2 = 3$. Denote the line bundle induced by $V$ on $\tilde{C}$ by $L_0$. Now we prove the result for all odd degrees $d$. Pick an effective line bundle $M$ of degree $\frac{1}{2}(d - 2g -1)$ on $C$. Then $L_0 \otimes \psi^*M$ has two global sections and degree $d$. Furthermore its pushforward is $V\otimes \psi_*L$ and this is stable. 
The proof for even degree is completely analogous.
\end{proof}
\begin{corollary}\label{cor:closure_stable_locus}
Let $C$ be a general curve of genus $g\geq 2$. Let $\psi: \tilde{C} \to C$ be a spectral cover corresponding to a general element $(s_1,s_2)$ of the Hitchin base $H^0(K_C)\oplus H^0(K_C^2)$. If $d > 3g-3$, then $W^1_d(\tilde{C})$ equals the closure of $W^1_d(\tilde{C})_s$.
\end{corollary}
\begin{proof}
We can assume $d < g(\tilde{C}) +1 = 4g-2$ because otherwise $W^1_d(\tilde{C})$ equals $\op{Pic}^d(\tilde{C})$ by Riemann-Roch. By classical Brill-Noether theory every component of $W^1_d(\tilde{C})$ has dimension at least $\rho(g(\tilde{C}),1,d) = -4g + 1 -2d$. Note that the strata $W^{1,-1}_{d_1,d_2}(\tilde{C})$ have for $d_1 + d_2 + 2g-2 > 3g-3$ dimension less than $\rho(g(\tilde{C}),1,d)$. Hence the closure of any of the strata can not be an irreducible component. This also holds for $W^{0,0}_{d_1,d_2}(\tilde{C})$ and $W^1_d(\tilde{C})_{sss}$. Hence all irreducible components must be contained in the closure of $W^1_d(\tilde{C})_s$. 
\end{proof}
\begin{remark}
Note that the Brill-Noether number $\rho(g(\tilde{C}),1,d)$ equals $2d - 4g+1$ which is also the dimension of $W^1_d(\tilde{C})$. Hence this stratum behaves exactly as predicted by classical Brill-Noether theory. 
\end{remark}
\section{The Gonality of canonical covers}\label{section:gonality}
We apply strategy \ref{strategy} to canonical covers as defined in the introduction. 
All double covers between smooth curves are spectral covers \cite[section 0.2]{Cossecdolgachev}. Therefore canonical covers are defined by the data of a theta characteristic $\vartheta$ and sections $s_1 \in H^0(\vartheta),s_2 \in H^0(K_C)$. The branch divisor is cut out by the section $\omega = s_2 - \frac{1}{4}s_1^2$. Actually the curve $\tilde{C}$ is already determined by $C,\vartheta$ and $\omega$: it is isomorphic to $V(\tau^2 + \omega) \subset \Aff(\vartheta)$. As in the last section we follow strategy \ref{strategy}. Because we only care about the existence of $g^1_d$'s it suffices to give upper bounds on the dimension of $W^1_C(\theta)(2,d)$.

In describing the gonality of a canonical cover the cases of low genus are exceptional, so as a warm-up we treat the case $g =2$ and $3$.

\begin{proposition}\label{prop:g=2}
Let $C$ be a smooth curve of genus $g = 2$ and $\psi: \tilde{C} \rightarrow C$ a canonical cover. Write $\vartheta$ for the associated theta characteristic. 
If $\vartheta$ is even the gonality of $\tilde{C}$ is $3$.
Otherwise the gonality of $\tilde{C}$ is $2$ i.e it is a hyperelliptic curve.  
\end{proposition}
\begin{proof}
First take an odd theta characteristic so that $h^0(\vartheta) = 1$. We have $h^0(\psi^*\vartheta) = h^0(\vartheta) + h^0(\OO_C) = 2$ and we have found a $g^1_2$.
In the case of an even characteristic $h^0(\vartheta)=0$ holds. We first show that $\tilde{C}$ does not admit a $g^1_2$. Assume for contradiction that there is one and call it $L$. If $L$ is the pullback of a line bundle of degree one on $C$ its pushforward is of the form $M \oplus (M \otimes \vartheta^\vee)$. If $M \otimes \vartheta^\vee$ were effective it would be the trivial line bundle. But then $M$ is $\vartheta$ itself and we have $h^0(\psi_*L) = h^0(\vartheta) + h^0(\OO_C) = 1$. Hence $L$ can not come from pullback of a bundle on $C$. Let $M$ be a sub line bundle of $\psi_*L$ of maximal degree. From Lemma \ref{lemma:max_subbundle} we see that $\deg M < \frac{1}{2}(\deg \psi_* L + 1) = 1$. Because $h^0(L)$ is nonzero there is an injection $\OO_C \rightarrow \psi_*L$ and we can take $M = \OO_C$. Write $M' = \psi_*L / \OO_C$. It is a line bundle of degree $1$ and since $h^0(\psi_*L) =2 $ it is effective. We have two cases:

\textbf{The pushforward splits:} $\psi_*L \cong \OO_C \oplus M'$. Because $\vartheta$ has only one global section the twisted endomorphism has the form $\phi =  \begin{pmatrix}
0 & \phi_{12} \\
\phi_{21} & 0  
\end{pmatrix}$. The entry $\phi_{21}$ is nonzero and lives in $H^0( (M')^\vee\otimes \vartheta)$. Hence $M' \cong \vartheta$. This is not effective and we have a contradiction. 

\textbf{The extension $\psi_*L$ of $M'$ by $M$ is non split.} We apply Lemma \ref{lemma:ext} to obtain the contradiction $h^0(\psi_*L) < 2$. The multiplication map $H^0(K_C) \otimes H^0(M') \rightarrow H^0(K_C \otimes M')$ is quickly seen to be injective and both sides are two dimensional hence the requirements of the Lemma are met.
\end{proof}

\begin{proposition}\label{prop:g=3}
Let $(C,\vartheta,s_1,s_2)$ be the data of a canonical cover $\psi: \tilde{C} \rightarrow C$ and assume $C$ is of genus $3$ and not hyperelliptic. Assume the branch divisor is a general canonical divisor. If $\vartheta$ is even the gonality of $\tilde{C}$ is $5$. Otherwise $\vartheta$ is odd and $\tilde{C}$ has gonality $4$.  
\end{proposition}
\begin{proof}
\textbf{Take $\vartheta$ to be even}. The genus of $\tilde{C}$ is $7$. By the existence theorem of Brill Noether theory $\tilde{C}$ admits a $g^1_5$. Hence for an even theta we need to show that there is no $g^1_4$ on $\tilde{C}$. If $L$ is in $\op{Pic}^4(C)$ its pushforward to $C$ has degree $2$. We construct $W^1_C(\vartheta)(2,2)$. A rank $2$ vector bundle of degree $2$ with more than $1$ linearly independent global sections can not be stable. This is proved in a note after the Theorem in \cite{Teixidor}. Hence we consider Higgs bundles  $(E, \phi)$ whose vector bundle is semistable or unstable. We first take $E$ to be semistable. Then $E$ is in $\op{Ext}^1(M',M)$, where $M$ and $M'$ are effective line bundles of degree $1$. They are not isomorphic because by Lemma \ref{lemma:max_subbundle} $M^\vee \otimes M' \otimes \vartheta$ is effective. If $E$ does not split the requirements of Lemma \ref{lemma:ext} are fulfilled and $h^0(E) < 2$. Therefore we may assume $E = M \oplus M'$. The twisted endomorphism $\phi$ has the form
\begin{equation*}
\phi = \begin{pmatrix}
0 & \phi_{12} \\
\phi_{21} & 0 
\end{pmatrix}: M \oplus M' \rightarrow (M  \oplus M') \otimes \vartheta.
\end{equation*}
By Riemann-Roch we have $h^0(M'\otimes \vartheta) = h^0(M\otimes \vartheta) = 1$. We now have constructed a three dimensional family of Higgs bundles, because we have one degree of freedom in choosing $M$ and $M'$. But observe that all stable Higgs bundles with underlying vector bundle $M\oplus M'$ are isomorphic. Hence our family has only a two dimensional image in $W^1_C(\vartheta)(2,2)$.

Next we look at an unstable $\psi_*L$. Let $M$ be the maximal destabilizing line bundle. By Lemma \ref{lemma:max_subbundle} we have $\deg M < \op{gon}(C)$. Hence $h^0(M) \leq 1$. The quotient $M' = \psi_*L/M$ has degree at most zero and is effective. This implies $M' \cong \OO_C$ and $M^\vee \otimes M' \otimes \vartheta$ is not effective. This is a contradiction to Lemma \ref{lemma.twisted_bound} and we have shown $\dim W^1_C(\vartheta)(2,2) = 2 < h^0(K_C)$.
\medskip

\textbf{If $\vartheta$ is odd} $\psi^*\vartheta$ is a $g^1_4$ by the projection formula. We prove that $\tilde{C}$ is not trigonal: By equation (\ref{eq:push_and_det}) we need to analyze $W^1_C(\vartheta)(2,1)$. Take $(E,\phi)$ to be a Higgs bundle in this locus. The vector bundle $E$ is quickly seen to be an extension of $\OO_C$ by a line bundle $\OO_C(p), p \in C$.  Applying Lemma \ref{lemma:ext} shows $E \cong \OO_C \oplus \OO_C(p)$. Its space of twisted endomorphisms has dimension
\begin{equation*}
h^0(\End(E)\otimes \vartheta) =  2h^0(\vartheta) + h^0(\vartheta(p)) + h^0(\vartheta(-p)). 
\end{equation*}
This dimension is $3$ for a general point. Only if $p$ is  $p_1$ or $p_2$, where $p_1 + p_2$ is the effective divisor of $\vartheta$, the dimension is $4$.  The automorphism group of $E$ is three dimensional: Its elements $\rho$ can be written as block matrices as follows: 
\begin{equation}
\rho = \begin{pmatrix}
\rho_{11} & 0 \\
\rho_{21} & \rho_{22}
\end{pmatrix}, \;\rho_{11},\rho_{22} \in \Bbbk^\times, \rho_{21} \in H^0(\OO_C(p)). 
\end{equation}
It acts by conjugation on the space of twisted endomorphisms and the orbits are two dimensional.  We conclude that 
\begin{equation}
\dim W^1_C(\vartheta)(2,1) = 4 - \dim \op{Aut}(E) + 1 = 2.
\end{equation}
This is smaller than $h^0(\vartheta) + h^0(K_C)$ and we are done.
\end{proof}

After these two special cases we come to the following more precise formulation of the first part of Theorem C:
\begin{theorem}\label{thm:gon_prec}(Theorem C (i))
Let $C$ be a smooth curve of genus $g$ and take $\psi: \tilde{C} \rightarrow C$ to be a canonical divisor whose branch divisor is general. Write $\vartheta$ for the associated theta characteristic. Then we have
\begin{itemize}
\item[(i)] If $g$ is even, $g \geq 4$ and $\vartheta$ is even the gonality of $\tilde{C}$ is $g+2$.
\item[(ii)] If $g$ is even, $g \geq 8$ and $\vartheta$ is odd, the gonality of $\tilde{C}$ is $g+2$.
\item[(iii)] If $g$ is odd, $g \geq 8$ and $\vartheta$ is even, the gonality of $\tilde{C}$ is $g+3$.
\item[(iv)] If $g$ is odd, $g \geq 11$ and $\vartheta$ is odd, the gonality of $\tilde{C}$ is $g+3$.
\end{itemize}
\end{theorem}
\begin{proof} We go through the four cases separately:
\medskip

\textbf{Even genus and even $\vartheta$}: The gonality of $C$ is $k+1$ where $g = 2k$ by the Brill-Noether theorem. We show that a line bundle $L$ on $\tilde{C}$ of degree $2k+1$ is not a pencil. Its pushforward has degree $2$. We have shown in Proposition \ref{prop:g=3} $\dim W^1_C(\vartheta)(2,2) = 2$. Because we never used the assumption $g = 3$ in the proof this statement still holds true.
\medskip 

\textbf{Even genus and odd $\vartheta$}:
Again we analyze a Higgs bundle $(E,\phi)$ in $W^1_C(\vartheta)(2,2)$. 
Either $E$ is semistable or unstable. Take it to be semistable. It is an extension of two line bundles $M,M'$, both of degree $1$. If $M$ and $M'$ are not isomorphic the pushforward splits by Lemma \ref{lemma:ext}. Here we have the choice of $M$, $M'$ and of a twisted endomorpism. This gives a $4$ dimensional family of Higgs bundles. Next we assume $E$ is a nontrivial extension of $M$ by $M$. If $M$ is effective the multiplication map $\mu: H^0(M) \otimes H^0(K_C\otimes M^\vee) \rightarrow H^0(K_C)$ has a one dimensional cokernel. Therefore $ K_{M,M}$ has dimension $1$. Hence up to automorphism we need to look at one nonsplit extension. From Lemma \ref{lemma.twisted_bound} we get $h^0(\End(E) \otimes \vartheta) \leq 4h^0(\vartheta) = 4$. Taking the choice of the line bundle $M$ into account we find another $5$ dimensional family of Higgs bundles. 

Next consider an unstable $E$. Represent it as an extension of a degree $0$ line bundle $M'$ by a degree $2$ line bundle $M$. To have two linear independent global sections we need to have effective $M$ and $M'$. Clearly $M'$ is the trivial line bundle. For $M$ one has finitely many choices because it needs to hold that $M\otimes \vartheta$ is effective by Lemma \ref{lemma:max_subbundle}. Let $D = p_1 + \ldots + p_{g-1}$ be the unique divisor of $\vartheta$. There are $i <j$ such that $M \cong \OO_C(p_i + p_j)$. The extensions needs to split by Lemma \ref{lemma:ext}. From Lemma \ref{lemma.twisted_bound} we get 
\begin{equation*}
h^0(\End(E) \otimes \vartheta) \leq 2h^0(\vartheta) + h^0(M \otimes \vartheta) + h^0(M^\vee \otimes \vartheta) = 2 + 3 +1 = 6.
\end{equation*}
Hence we have two $6$ dimensional families of Higgs bundles. 
If the genus is $\geq 7$ the spectral curves of these Higgs bundles are not general.
\medskip 

\textbf{Odd genus and even $\vartheta$}:
The gonality of $C$ is $k +2$, where $g = 2k+1$. We show that every line bundle $L$ of degree $2k +3$ has less than two linear independent global sections. Thus we look at Higgs bundles $(E,\phi)  \in W^1_C(\vartheta)(2,3)$. First assume $E$ is stable. The collection of stable vector bundles $E$ on $C$ with $h^0(E) = 2$ is proved to be three dimensional in the Theorem in \cite{Teixidor}. We exhibit $E$ as an extension of line bundles. Our vector bundle has a subbundle $M$ of degree $1$: Take $s_1,s_2 $ a basis of $H^0(\psi_*L)$. The sections cannot be linear independent in each fiber because they would otherwise define an isomorphism $\psi_*L \cong \OO_C^{\oplus 2}$. Hence there are $\lambda_1,\lambda_2$ such that $\lambda_1s_1 + \lambda_2s_2$ has a zero. We get an injection $\OO_C(p) \rightarrow \psi_*L$ and by stability $\OO_C(p)$ is a line subbundle. Write $M' = \psi_*L/M$. We use the bound from Lemma \ref{lemma.twisted_bound}:
\begin{equation*}
h^0(\End(E)\otimes \vartheta) \leq h^0(M^\vee \otimes M' \otimes \vartheta) + h^0((M')^\vee \otimes M \otimes \vartheta) \leq 2 +1 
\end{equation*}
Therefore we get the following dimension count for the Higgs bundles whose underlying vector bundle is stable and has two linear independent sections:
\begin{equation}
3 + h^0(\End(E)\otimes \vartheta) \leq 6.
\end{equation} 
Now consider an unstable $E$. It is an extension of $M'$ by $M$. The line bundles have degree $1$ and $2$ respectively and are effective by Lemma \ref{lemma:max_subbundle}.  Write $M = \OO_C(p_1 + p_2), M' = \OO_C(p_3)$. The points $p_i$ are uniquely determined and $p_3$ is different from $p_1$ and $p_2$ because $M' \otimes M^\vee \otimes \vartheta$ is effective. We have two degrees of freedom in choosing these three points. A quick Riemann-Roch calculation shows that the multiplication map in Lemma \ref{lemma:ext} is surjective. We conclude $\psi_*L \cong M \oplus M'$. We have two degrees of freedom in choosing $M$ and $M'$. The collection of these Higgs bundles has therefore dimension at most:
\begin{equation}
2 + h^0(\End(E) \otimes \vartheta) \leq 2 + 3  = 5. 
\end{equation}
\medskip

\textbf{Odd genus and odd $\vartheta$}:
As in the last case we have $E$ is of degree $3$. The arguments are almost identical but the bounds are higher. Start with $E$ stable. Lemma \ref{lemma.twisted_bound} shows here:
\begin{align*}
h^0(\End(E) \otimes \vartheta) &\leq 2h^0(\vartheta) + h^0(M^\vee \otimes M' \otimes \vartheta) + h^0((M')^\vee \otimes M \otimes \vartheta) \\
&\leq 2 + 2 + 3=7.
\end{align*}
The dimension bound is therefore 
\begin{equation*}
3 + h^0(\End(E)\otimes \vartheta) \leq 10.
\end{equation*}
Now we treat an unstable pushforward. Define $M$ and $M'$ and the points $p_i$ as in the third case. If $p_3 \notin \{p_1,p_2\}$ the extension splits. We get the dimension count for the corresponding Higgs bundles:
\begin{align*}
2 + h^0(\End(E) \otimes \vartheta) &\leq 2 + 2h^0(\vartheta) +h^0(M^\vee \otimes M'\otimes \vartheta) + h^0((M')^\vee\otimes M \otimes \vartheta)\\ &\leq 2 + 2 + 2 +3= 9.
\end{align*}
There are also a onedimensional family of such vector bundles in which $p_3 \in \{p_1,p_2\}$ holds and the extension of $M'$ by $M$ does not split. We get the dimension bound for the corresponding family.
\begin{align*}
1 + h^0(\End(E) \otimes \vartheta) &\leq 1 + 2h^0(\vartheta) +h^0(M^\vee \otimes M'\otimes \vartheta) + h^0((M')^\vee\otimes M \otimes \vartheta) \\
&\leq 1 + 2 + 1 + 2 = 6.
\end{align*}
This concludes the last case in the proof.
\end{proof}
\section{The rest of the Gonality Sequence}\label{section:gon_seq}
We will prove Theorem C via an inequality involving the Brill-Noether number $\rho(g,r,d) := g -(r+1)(g+r-d)$. We use the assumption $r \ll g$ to make the proof more illuminating.  

\begin{proof} Take $L$ to be a $g^r_d$, which does not lie in $\psi^*\op{Pic}(C)$. Assume for contradiction that $d = 2d_r$ holds. Let $\varepsilon$ be a positive number which we can choose later on as small as we want and assume $(r+1)^2/g < \varepsilon$. There is an injection of $\OO_C$ into $\psi_*L$. Therefore we find an effective line subbundle $M$ of $\psi_*L$. Its degree is by Lemma \ref{lemma:max_subbundle} smaller than $d_r$. Write $M' = \psi_*L /M$ and observe that $s = h^0(M) -1$ is nonnegative. The Brill-Noether theorem implies $\rho(g,s,\deg M) \geq 0$. Then
\begin{equation*}
\deg M \geq \frac{sg}{s+1} + s
\end{equation*}
follows. Observe that $h^0(M') \geq r-s$. We similarly get $\deg M' \geq \frac{(r-s-1)g}{r-s} + r- s-1$. On the other hand the minimality of $d_r$ implies $\rho(g,r,d_r-1) < 0$ and $\rho(g,r,d_r) < r+1$. This translates into
\begin{equation*}
d_r < \frac{rg}{r+1} +r+1 \leq \frac{(r + \varepsilon)g}{r+1}. 
\end{equation*}
We combine these inequalities to obtain 
\begin{align*}
\deg L &= \deg \psi_*L + g-1 \\&= \deg M + \deg M' + g -1  \\&\geq \frac{sg}{s+1} + \frac{(r-s-1)g}{r-s} + g-2 + r\\ & \geq 2\frac{(r + \varepsilon)g}{r+1} > 2d_r. 
\end{align*}
The second to last inequality is seen as follows: $\frac{(r-1)g}{r+1}$ lies between $\frac{sg}{s+1}$ and $\frac{(r-s-1)g}{r-s}$ and $\frac{(r+1)g}{r+1} < g - 2 +r$. For sufficiently small $\varepsilon$ we have derived a contradiction to our assumption $d_r = 2r$. The Theorem is proven.
\end{proof}
\section{Equations for specific Canonical Covers}\label{section:equations}

Let $(C,\vartheta,\omega)$ be the data for a canonical cover. One way to get a projective coordinate ring for $\tilde{C}$ is to start with the section ring $R_{\vartheta} = \bigoplus_{n=0}^\infty H^0(C,\vartheta^n)$ and try to understand the section ring $R_{\psi^*\vartheta}$. By the projection formula one obtains the following representation of it:
\begin{equation}\label{eq:can_coord_ring}
R_{\psi^*\vartheta} \cong R_{\vartheta}[x]/(x^2 - \omega).
\end{equation}  
The image of $x$ has degree $1$ in the above ring. The canonical coordinate ring of $\tilde{C}$ is the following subring of $R_{\psi^*\vartheta}$:
\begin{equation*}
\bigoplus_{m=0}^\infty H^0(\tilde{C},K_{\tilde{C}}^m) = \bigoplus_{m=0}^\infty \left( R_{\psi^*\vartheta} \right)_{3m}.
\end{equation*}
\begin{example}
We apply this in the case of genus $2$ curves. The curve $C$ is then hyperelliptic. The Theta characteristics of $C$ are complety understood, see for example section 5.2 in \cite{Dolgachev}. It comes with a double cover to $\PP^1$. After choosing coordinates $X,Y$ on $\PP^1$ we write $B(X,Y)$ for the polynomial cutting out the branch divisor. Let $p_1,\ldots,p_6$ be the branch points. If $\vartheta$ is an even theta characteristic then after possibly renumbering the branch points we have $\vartheta^3 \cong \OO_C(p_1+p_2+p_3)$. We find sections $w_1,w_2$ of this line bundle whose divisors are $p_1+ p_2 +p_3$ and $p_4+p_5+p_6$ respectively. Let us factorize $B = B_1B_2$ such that the roots of $B_1$ are $p_1,p_2$ and $p_3$. Then we have
\begin{equation*}
R_\vartheta \cong \Bbbk[\overset{2}{X},\overset{2}{Y},\overset{3}{W_1},\overset{3}{W_2}] / (W_1^2 - B_1,W_2^2 - B_2).
\end{equation*}
Here the numbers above the generators indicate their degree. 
To see this observe that the right hand side $S$ is isomorphic to the subring of $R_{\vartheta}$ generated by $\Bbbk[X,Y]$ and $w_1,w_2$. Next we find bases of the graded pieces of $S$: 
\begin{align*}
S_{2k} &= \langle X^2iY^{2(k-i)},X^2iY^{2(k-3-i)}W_1W_2\rangle_{i=0}^k, \quad k\geq 3, \\
S_{2k+1} &= \langle W_1X^iY^{k-1-i},W_2X^iY^{k-1-i}\rangle_{i=0}^k,\quad k\geq 1.
\end{align*}
One uses this to check that the Hilbert functions on both sides agree, which implies $R_{\vartheta} \cong R$.
In the next step we represent the canonical divisor as a linear form $L(X,Y)$. Then (\ref{eq:can_coord_ring}) becomes
\begin{equation*}
R_{\psi^*\vartheta} \cong \Bbbk[\overset{2}{X},\overset{2}{Y},\overset{3}{W_1},\overset{3}{W_2},\overset{1}{U}]/( W_1^2 - B_1,W_2^2 - B_2,U^2 - L).
\end{equation*}
A curve is hyperelliptic if and only if the canonical coordinate ring is not generated in degree $1$ by Noether's theorem. We now show that $ (R_{\psi^* \vartheta})_3 = \langle XU,YU,W_1,W_2\rangle$ generates the canonical coordinate ring. The canonical coordinate ring is always generated in degrees $1$ and $2$, hence it is enough to show that we can generate $(R_{\psi_*\vartheta})_6$ with degree $1$ elements. Using equation (\ref{eq:push_and_pull}) we find 
\begin{align*}
(R_{\psi_*\vartheta})_6 = \langle X^3,X^2Y,XY^2,Y^3,
W_1W_2,XW_1U,XW_2U,YW_1U,YW_2U\rangle.
\end{align*}
The monomials only involving $X$ and $Y$ can all lie in $\langle U^2X^2,U^2Y^2,U^2XY,W_1^2,W_2^2\rangle$.
This recovers one half of Proposition \ref{prop:g=2} for even theta characteristics. The other assertations we have shown there can also be proved by working with explicit equations as the reader may verify. 
\end{example}
We next indicate how to treat canonical covers of plane curves. We recall the description of theta characteristics given in \cite[Proposition 3.1 a)]{Beauville}: Assume $C$ is a smooth plane curve of degree $d \geq 4$. If $\vartheta$ is a noneffective theta characteristic on it it has a minimal resolution of the form 
\begin{equation}\label{eq:plane_theta}
\begin{tikzcd}[ampersand replacement=\&]
	0 \& {\mathcal{O}_{\mathbb{P}^2}(-2)^{\oplus d }} \& {\mathcal{O}_{\mathbb{P}^2}(-1)^{\oplus d }} \& \vartheta \& 0
	\arrow[from=1-1, to=1-2]
	\arrow[from=1-3, to=1-4]
	\arrow[from=1-4, to=1-5]
	\arrow["M", from=1-2, to=1-3]
\end{tikzcd}
\end{equation}
where $M = (m_{ij})$ is a symmetric matrix of linear forms. The determinant of $M$ cuts out exactly $C$. By the adjunction formula $K_C \cong \OO_C(d-3)$ holds. We use the minimal resolution to describe the multiplication map 
\begin{equation*}
H^0(\vartheta(1)) \otimes H^0(\vartheta(1)) \rightarrow H^0(K_C(2)) \cong H^0(\OO_C(d-1)).
\end{equation*}
We denote the matrix obtained from $M$ by leaving out the $i$-th and $j$-th column by $M_{ij}$. The resolution in (\ref{eq:plane_theta}) induces a basis of $H^0(\vartheta(1))$ which we denote by $(e_1,\ldots,e_d)$.
\begin{lemma}
For the multiplication of sections in $\vartheta(1)$ the following holds:
\begin{equation*}
e_ie_j = (-1)^{i+j}\op{det}M_{ij}\in H^0(\OO_C(d-1))
\end{equation*}
\end{lemma}
\begin{proof}
The idea is to make the proof of \cite[Theorem B]{Beauville} more explicit in our case. We write a resolution of $\vartheta(1)$ as 
\begin{equation*}
\begin{tikzcd}[ampersand replacement=\&]
	0 \& {L_1\otimes \mathcal{O}_{\mathbb{P}^2}(-1)} \& {L_0\otimes \mathcal{O}_{\mathbb{P}^2}} \& {\vartheta(1)} \& 0
	\arrow[from=1-1, to=1-2]
	\arrow["M",from=1-2, to=1-3]
	\arrow[from=1-3, to=1-4]
	\arrow[from=1-4, to=1-5]
\end{tikzcd}.
\end{equation*}
Here $L_0$ and $L_1$ are $d$ dimensional vector spaces. From (\ref{eq:plane_theta}) we get bases $f_i$ and $e_j$ of $L_0$ and $L_1$ respectively. Applying $\Hom_{\PP^2}(-,\OO_{\PP^2}(-1))$ to the sequence we find
\begin{equation*}
\begin{tikzcd}[ampersand replacement=\&]
	0 \& {L_0^\vee\otimes \mathcal{O}_{\mathbb{P}^2}(-1)} \& {L_1^\vee\otimes \mathcal{O}_{\mathbb{P}^2}} \& {\Hom_{C}(\vartheta(1),\OO_C(d-1))} \& 0
	\arrow[from=1-1, to=1-2]
	\arrow["M", from=1-2, to=1-3]
	\arrow["\mu", from=1-3, to=1-4]
	\arrow[from=1-4, to=1-5]
\end{tikzcd}.
\end{equation*}
Here we used $\Ext^1_{\PP^2}(-,\OO_{\PP^2}(-1)) \cong \Hom_C(-,\OO_{C}(d-1))$ which follows from Grothendieck-Verdier duality. Denote dual bases with a star. We claim that $\mu(f_j^*) = \sum_i (-1)^{i+j} e_i^* \otimes \op{det}M_{ij}$. This element defines a homomorphism $\vartheta(1) \rightarrow \OO_C(3)$:
\begin{equation*}
\mu(f_j^*)(\sum_i m_{ij} e_i) = \sum_i (-1)^{i+j} m_{ij} \op{det} M_{ij} = \op{det} M = 0.
\end{equation*}
We used Laplace expansion and the fact that $\op{det}M$ defines $C$. The composition $\mu \circ M$ is zero by Laplace expansion as well:
\begin{equation*}
\mu(\sum_j m_{ij}f_j^*) = \sum_i \sum_j (-1)^{i+j}e_i^* \otimes m_{ij}\det M_{ij}.
\end{equation*}
This suffices to identify the natural boundary map with $f_j^* \mapsto \sum_i (-1)^{i+j} e_i^* \otimes \op{det}M_{ij}$.
\end{proof}
\begin{example}
A nonhyperelliptic curve $C$ of genus $3$ is a plane quartic.  We derive from the last Lemma a presentation of the section ring of an even theta characteristic $\vartheta$:
\begin{equation*}
R_{\vartheta} =  \Bbbk[\overset{2}{X_1},\overset{2}{X_2},\overset{2}{X_3},\overset{3}{e_1},\ldots,\overset{3}{e_4}]/(\sum_i m_{ij}e_i, e_ie_j - (-1)^{i+j}\op{det}M_{ij}). 
\end{equation*}
Represent the differential defining the branch divisor by a linear form $L(X_1,X_2,X_3)$. We have from (\ref{eq:can_coord_ring})
\begin{equation*}
R_{\psi^*\vartheta} =  \Bbbk[\overset{2}{X_1},\overset{2}{X_2},\overset{2}{X_3},\overset{3}{e_1},\ldots,\overset{3}{e_4},\overset{1}{U}]/(\sum_i m_{ij}e_i, e_ie_j - (-1)^{i+j}\op{det}M_{ij},U^2 - L). 
\end{equation*}
In Proposition \ref{prop:g=3} we have proved that a general canonical cover $\tilde{C} $ of $C$ does not admit a $g^1_4$, but there are examples where the gonality is $4$. We will indicate how to verify this fact with computer algebra. To construct the example we need to give a twisted endomorphism of $E = \OO_C(p) \oplus \OO_C(q)$ where $p$ and $q$ are points on $C$ such that $\vartheta(q-p)$ is effective. To find these points construct a section $r \in H^0(\vartheta(1))$ whose divisor has the form
\begin{equation}
\div r = p_1 + p_2 + p_3 + q_1 + q_2 + q_3 
\end{equation} 
where $q_1,q_2,q_3$ and $q_4$ are colinear points on $C$. Then $\vartheta(q_4 - p_1) \cong \OO_C(p_2 + p_3)$, hence we may put $p = p_1, q = q_4$. Let us take $q_1 = [1,0,0],q_2 = [0,1,0],q_3 = [1,1,0]$ and 
\begin{equation*}
M = \begin{pmatrix}
X_1 + X_2 & X_2 & X_2 & X_1 - X_2\\
X_2 & X_0 + X_2 & X_2 & X_0 - X_2 \\
X_2 & X_2 & X_0 - X_1 + X_2 & X_0 - X_1 + X_2 \\
X_1 - X_2 & X_0 - X_2 & X_0 - X_1 + X_2 & X_1
\end{pmatrix}
, \vec{r} = \begin{pmatrix}
0 \\ 0 \\ 0 \\ 1
\end{pmatrix}.
\end{equation*}
The determinant of $M$ defines a smooth quartic. Furthermore the matrix $\tilde{M} = \begin{pmatrix}
M  |\vec{r}
\end{pmatrix}$ is singular when evaluated at the points $p_1,p_2$ and $p_3$. The six points at which $\tilde{M}$ is singular are $p_1,p_2,p_3,q_1,q_2$ and $q_3$. With the help of a computer one finds
\begin{equation*}
p_1 = [0,0,1],\; p_2 = [1160,796,1] ,\; p_3 = [794,1162,1].
\end{equation*} 
The section $r$ is then given as the image of $\vec{r}$ under the map $\OO_{\PP^2}^{\oplus 4} \rightarrow \vartheta(1)$ from the resolution of $\vartheta(1)$. 
Set $\phi = \begin{pmatrix} 0 & s_1 
\\ s_2 & 0 \end{pmatrix}$, where $s_1 \in H^0(\vartheta(p -q)), s_2 \in H^0(\vartheta(q-p))$. The Higgs bundle $(E,\phi)$ defines a $g^1_3$ on the canonical cover with data $(C, \vartheta,\omega)$. Here $\omega$ is an abelian differential cutting out the line through $p_2$ and $p_3$. The canonical ring of $\tilde{C}$ is the subring of $R_{\psi^*\vartheta}$ generated by all elements of degree $3$. Computations with the computer algebra systems Oscar \cite{OSCAR} and Macaulay2 \cite{M2} show that our example and canonical covers with a general branch divisor have canonical rings with different Betti numbers. This is to be expected because of Green's conjecture and them having different Clifford indices.    
\end{example}

\bibliography{bibliography} 
\bibliographystyle{siam}
\noindent Clemens Nollau,
Universit\"at T\"ubingen
\hfill {\tt clemens.nollau@uni-tuebingen.de}

\end{document}